\documentclass[reqno]{amsart}

    \usepackage{hyperref}
    \usepackage{appendix}
    \usepackage{float}
    \usepackage{amsmath}
    \usepackage{graphicx}
    \usepackage{latexsym}
    \usepackage{amsfonts} 
    \usepackage{amssymb}
    \usepackage{verbatim}
    \usepackage{color}
    \usepackage{todonotes}
    \usepackage{cleveref}
    \theoremstyle{plain}
    \newtheorem{theorem}{Theorem}
    \newtheorem{corollary}[theorem]{Corollary}
    \newtheorem{lemma}[theorem]{Lemma}
    \newtheorem{proposition}[theorem]{Proposition}
    \theoremstyle{definition}
    \newtheorem{assumption}{Assumption}

    \newtheorem{remark}[theorem]{Remark}
    \newtheorem*{remark*}{Remark}

    \newcommand{\R}{\mathbb R}
    
    \newcommand{\pr}{\mathbf P}
    \renewcommand{\P}{\mathbf P}
    \newcommand{\e}{\mathbf E}
    \newcommand{\E}{\mathbf E}
    \newcommand{\ind}{{\rm 1}}
    \newcommand\eps\varepsilon
    \newcommand{\N}{{\mathbb{N}}}
    \newcommand{\EW}[1]{\mathbf{E}\left[#1\right]}

\begin{document}
    \title[Random walks with square-root boundaries]
    {Random walks with square-root boundaries:\\
    the case of exact boundaries $g(t)=c\sqrt{t+b}-a$} 
    \thanks{S. Terveer was partially supported by the Deutsche Forschungsgemeinschaft (DFG, German Research Foundation) – TRR 352 “Mathematics of Many-Body Quantum Systems and Their Collective Phenomena" – Project-ID 470903074.\\
    V. Wachtel was partially supported by DFG}
    \author[Denisov]{Denis Denisov}
    \address{Department of Mathematics, University of Manchester, Oxford Road, Manchester M13 9PL, UK}
    \email{denis.denisov@manchester.ac.uk}

    \author[Sakhanenko]{Alexander Sakhanenko}
    \address{Sobolev Institute of Mathematics, Russia}
    \email{aisakh@mail.ru} 
    
    \author[Terveer]{Sara Terveer}
    \address{Mathematisches Institut, Ludwig-Maximilians-Universität München, Theresienstrasse 39, 80333 München, Germany}
    \email{terveer@math.lmu.de}
    
    \author[Wachtel]{Vitali Wachtel}
    \address{Faculty of Mathematics, Bielefeld University, Germany}
    \email{wachtel@math.uni-bielefeld.de}

    \begin{abstract}
    	Let $S(n)$ be a real valued random walk with i.i.d. increments which have zero mean and finite variance.
        We are interested in the asymptotic properties of the stopping time 
        $T(g):=\inf\{n\ge1: S(n)\le g(n)\}$, where $g(t)$ is a boundary function.
        In the present paper we deal with the 
        parametric family of boundaries   $\{g_{a,b}(t)=c\sqrt{t+b}-a, b\ge0, a>c\sqrt{b}\}$.
        First, assuming that sufficiently many moments of increments of the  walk are finite,  we construct a positive space-time harmonic function $W(a,b)$.   
        Then we show that there exist $p(c)>0$ and a constant  $\varkappa(c)$ such that
        $\P(T_{g_{a,b}}>n)\sim \varkappa(c)\frac{W(a,b)}{n^{p(c)/2}}$ as $n\to\infty$.
\end{abstract}
    
    
    \maketitle

\section{Introduction, main results and discussion}
Consider a one-dimensional random walk
$$
S(0)=0,\quad S(n)=X_1+X_2+\ldots+X_n,\ n\ge1,
$$
where $X_1,X_2,\ldots$ are independent copies of a random variable $X$. We shall always assume that the random variable $X$ has zero mean and unit variance, $\E X=0$ and $\E X^2=1$. The main purpose of this paper is to study the behaviour of $\{S(n)\}$ killed at crossing a square-root boundary.
More precisely, for a function $g(t)$ we define the stopping time 
$$
T(g):=\inf\{n\ge1: S(n)\le g(n)\}.
$$
We shall always assume that $g(0)<0$ and that there exists a constant $c\in\R$ such that
\begin{equation}
\label{eq:g-assump}
g(t)\sim c\sqrt{t}\quad\text{as }t\to\infty.
\end{equation}
A quite important role in the study of the stopping time $T_g$ with general $g$ satisfying \eqref{eq:g-assump} is played by the following parametric family of boundary functions. For $b\ge 0$ and $a>c\sqrt{b}$ we set 
\begin{equation}
\label{eq:g-assump.a.b}
g_{a,b}(t)=c\sqrt{t+b}-a,\ t\ge0.
\end{equation}
In this case we set
$$
T_{a,b}:=T(g_{a,b})=\inf\{n\geq 1:a+S(n)\leq c\sqrt{n+b}\}.
$$
For this special family of boundaries we can interpret $a$ as a starting position of the walk and $b$ as a time shift. 
Using this interpretation of parameters $a$ and $b$ we may define a more general family of stopping times. For every boundary function $g$ we set 
$$
T_{a,b}(g):=\inf\{n\ge1: a+S(n)\le g(n+b)\},\quad a>g(b).
$$
Thus, $T_{a,b}=T_{a,b}(c\sqrt{\cdot})$ and, consequently, we deal with the exact boundary $c\sqrt{t}$ modulo temporal and spatial shifts.

First-passage problems related to stopping times $T(g)$ and $T_{a,b}$ can be seen as a one-sided analogue of the problem with two boundaries:
$$
T(g_1,g_2):=\inf\{n\ge N: S(n)\notin(g_1(n),g_2(n))\},
$$
where $N\ge1$, $g_i(t)\sim c_i\sqrt{t}$ with $c_1\le0\le c_2$, $c_1<c_2$. 
(The additional parameter $N$ is needed here, since the interval $(g_1(n),g_2(n))$ can be empty for small values of $n$.)
Apparently, the study of such two-sided first-passage problems has been initiated by Blackwell and Freedman~\cite{BF64}.
They have considered first time when the absolute values of a symmetric simple random walk exceeds $c\sqrt{n}$ and have shown that this random time has finite expectation if and only if $c<1$. Motivated by this surprising observation Breiman~\cite{Breiman65} has investigated the tail behaviour of $T_{g_1,g_2}$
in the case when $-g_1(t)=g_2(t)=c\sqrt{t}$. Assuming that the third moment of $X$ is finite he has shown that 
$$
\P(T(g_1,g_2)>n)\sim \alpha n^{-\gamma(c)},
$$
where $\gamma(c)$ does not depend on the distribution of the walk.  Breiman's result was later improved by Greenwood and Perkins~\cite{GP83}, who have shown that if $\E X^2\log(1+|X|)$ is finite then there exist a positive number $p(c_1,c_2)$ and a slowly varying function $L_{g_1,g_2}$ such that 
$$
\P(T(g_1,g_2)>n)\sim\frac{L_{g_1,g_2}(n)}{n^{p(c_1,c_2)/2}}
\quad\text{as }n\to\infty.
$$
Moreover, they have proven a limit theorem for $S(n)$ conditioned on
$\{T_{g_1,g_2}>n\}$. 

The corresponding picture in the case of one-sided boundaries is not as complete as in the two-sided case. Greenwood and Perkins~\cite{GP83} have shown that if $\E X^2\log(1+|X|)<\infty$ and if $g(t)\sim c\sqrt{t}$ with some $c\le0$, then there exist $p(c)>0$ and a slowly varying function 
$L_g$ such that 
\begin{equation}
\label{eq:GP}
\P(T(g)>n)\sim\frac{L_g(n)}{n^{p(c)/2}}.
\end{equation}
As in the two-sided case, the authors prove also a limit theorem for $S(n)$ conditioned on $\{T(g)>n\}$. Later it was shown in \cite{DSW19} that \eqref{eq:GP}
remains valid without extra moment assumption in the case $c=0$. Moreover, that paper contains sufficient and necessary conditions on boundaries $g(t)$ under which $g(t)$ is asymptotically constant.  The case $c=0$, i.e. $g(t)=o(\sqrt{t})$ is probably the most studied case of first-passage times over moving boundaries, and we refer to \cite{DSW19} for a rather detailed overview of existing results.

To the best of our knowledge, there is no version of \eqref{eq:GP} in the case when \eqref{eq:g-assump} holds with some $c>0$. Furthermore, as far as we know, there was no progress after \cite{GP83} in studying exit times
$T(g)$ and related problems in the case $c\neq0$ despite the fact that  walks with square-root boundaries play a significant role in various statistical and learning settings; see, for example, a recent paper Harvey et al \cite{HLPR}.

The main purpose of the present paper is to develop a technique which will allow one to study the tail behaviour of $T_{a,b}$ for boundaries $g_{a,b}(t)$ defined in \eqref{eq:g-assump.a.b} with negative as well as positive values of $c$. The case of general boundaries $g(t)$ satisfying \eqref{eq:g-assump}  will be studied in the follow-up paper~\cite{DSTW25}.

We shall take a route via potential theory for killed random walks and construct appropriate harmonic functions.
In the most classical case of constant boundaries a positive harmonic function describes, on one hand, the dependence of the distribution of the exit time on the initial state. On the other hand, harmonic functions allow one to perform the corresponding Doob $h$-transform and to introduce the conditioned process. In the case of square-root boundaries one expects that the same role will be played by the so-called space-time harmonic functions. To give reasons for that we describe next some known results for the Brownian motion. 

For a Brownian motion $B(t)$ killed at crossing $g_{a,b}$ such functions have been constructed by Novikov \cite{Novikov71}. 
Using these functions he obtained explicit expressions for moments of exit times. 
We now define that functions, since they will play a crucial role in our analysis of discrete time random walks. For every $p\in\R$ we set
\begin{align}                                                   \label{eq:Vdef.2}
\psi_p(x):=e^{x^2/4}D_{p}(x)
\quad\text{and}\quad
V_p(x,t):=t^{p/2}\psi_p\left(\frac{x}{\sqrt{t}}\right),
\quad t>0,
\end{align}
where $D_p(z)$ is the parabolic cylinder function. We also set 
\begin{align}
\label{eq:Vdef.1}
V_p(x,0)=x^p
\quad\text{for}\quad x>0.
\end{align} 
Novikov has shown that 
the processes $V_p(a+B(t),b+t)$ are martingales.  

Let $p(c)$ be the minimal positive root of the mapping $p\mapsto\psi_p(c)$. This implies that $V_{p(c)}(c\sqrt{t},t)=0$ and $V_{p(c)}(x,t)>0$ for all $x>c\sqrt{t}$. Combining this with the mentioned above martingale property of $V_{p(c)}$, we conclude that 
\begin{equation}\label{eq:cont.harm}
V_{p(c)}(a,b)=\E[V_{p(c)}(a+B(t),b+t);T_{a,b}^{(bm)}>t],\quad t>0,
\end{equation}
where 
$$
T_{a,b}^{(bm)}=\inf\{t>0:a+B(t)\le c\sqrt{b+t}\}.
$$
In other words, $V_{p(c)}$ is a positive {\it space-time harmonic}
function for $B(t)$ killed at $T_{a,b}^{(bm)}$. Equivalently, the process
$V_{p(c)}(a+B(t),b+t)\ind_{\{T_{a,b}^{(bm)}>t\}}$ is a positive martingale.
We note also that the space-time harmonicity is equivalent to the 'standard' harmonicity  for the two-dimensional process $(B(t),t)$ in the sense 
\begin{align*}
&\frac{\partial V_{p(c)}(x,t)}{\partial t}+\frac{1}{2}\frac{\partial^2 V_{p(c)}(x,t)}{\partial x^2}=0, \\
& V_{p(c)}(c\sqrt{t},t)=0,
\end{align*}
see~\eqref{eq:V-2deriv} below. 
The importance of the space-time harmonic function $V_{p(c)}$ becomes clear from the relation
\begin{equation}\label{eq:uchiyama}
\P(T_{a,b}^{(bm)}>t)\sim \varkappa(c)\frac{V_{p(c)}(a,b)}{t^{p(c)/2}}
\quad\text{as }t\to\infty,
\end{equation}
where $\varkappa(c)>0$. This relation has been shown by Uchiyama~\cite{U80}.

Our first result deals with space-time harmonic functions for discrete time random walks killed at down-crossing of the boundary $g_{a,b}$.
\begin{theorem}
\label{thm:space-time-h}
Assume that $\E |X|^{2+\delta}<\infty$ for some $\delta>0$ and
$\E |X|^{p(c)}<\infty$, where $p(c)$ is the minimal positive root of
$p\mapsto\psi_p(c)$. Then the function
$$
W(a,b):=\lim_{n\to\infty}\E[V_{p(c)}(a+S(n),b+n);T_{a,b}>n]
$$
is well-defined and satisfies
\begin{equation}\label{eq:sth1}
W(a,b)=\E[W(a+S(n),b+n);T_{a,b}>n],\quad n\ge1.
\end{equation}
Moreover, the mapping $a\mapsto W(a,b)$ is monotone increasing
and
$$
W(x,t)\sim V_{p(c)}(x,t)\quad \text{as }x,t\to\infty\text{ and }\frac{x}{\sqrt{t}}\ge c+\varepsilon \text{ with } \varepsilon>0.
$$
\end{theorem}
It is immediate from the definition of the stopping time $T_{a,b}$ that \eqref{eq:sth1} is equivalent to 
$$
W(a,b)=\E\left[W(a+S(n),b+n);
\min_{k\le n}\left(a+S(k)-c\sqrt{n+b}\right)>0\right].
$$
Thus, also for the function $W$, we can interpret the variable $a$ as a spatial (starting) point and $b$ as a time shift for the  boundary $c\sqrt{t}$. Equation~\eqref{eq:sth1} is a discrete-time analogue of~\eqref{eq:cont.harm} and it can be  viewed as a standard discrete harmonicity for 2-dimensional process $(S(n),n)$ killed at  $T_{a,b}$.

The space-time harmonic function is intrinsically connected to the stopping time $T_{a,b}$. We underline this connection by the following result.
\begin{corollary}
\label{cor:UBound}
Under the conditions of Theorem~\ref{thm:space-time-h},
$$
\P(T_{a,b}>n)\le M\frac{W(a,b)}{n^{p(c)/2}},\quad n\ge b,
$$
with some constant $M$ which is independent of $a$ and $b$.
(This constant may depend on $c$ and on the distribution of $X_1$.)
\end{corollary}

To prove Theorem~\ref{thm:space-time-h} we first construct a family of positive supermartingales for $S(n)$ killed at $T_{a,b}$. To this end we modify appropriately functions $V_p(x,t)$ for $p<p(c)$. (To understand, why $V_p$ with $p<p(c)$ is a good starting point, we look again at the Brownian motion. Recalling that
$V_p(a+B(t),b+t)$ is a martingale and that $V_p(c\sqrt{t},t)>0$ for every $p<p(c)$
and using the optional stopping theorem, we infer that
$V_p(a+B(t),b+t){\ind\{T_{a,b}^{(bm)}>t\}}$ is a supermartingale.) 
Applying the Optimal Stopping Theorem to that martingales, we derive suboptimal upper bounds for the tail of $T_{a,b}$, see Lemma~\ref{lem:Tbound} below. Having that bound we construct $W(a,b)$
by applying the strategy from \cite{DW15}, where harmonic functions for multidimensional walks in cones have been constructed.

Construction of supermartingales in the proof of Theorem~\ref{thm:space-time-h} is lengthy but uses only quite elementary calculations.  A disadvantage of this approach is a slightly stronger moment condition $\E|X|^{2+\delta}<\infty$. 
This moment condition will be relaxed in the follow-up paper~\cite{DSTW25}.

We now state our result on the tail behaviour of the stopping time $T_g$.
\begin{theorem}
\label{thm:tail}
Assume that all the conditions of Theorem~\ref{thm:space-time-h} are valid.
Then, for all fixed $a,b$ such that $a>c\sqrt{b}$,
\begin{equation}
\label{eq:asymp2}
\P(T_{a,b}>n)\sim\varkappa(c)\frac{W(a,b)}{n^{p(c)/2}},
\end{equation}
where $\varkappa(c)$ is the same constant as in~\eqref{eq:uchiyama}.
\end{theorem}

Comparing this with \eqref{eq:GP} we see that our result covers boundaries $g_{a,b}$ for all possible values of $c$. Furthermore, we show that a slowly varying function $L_g$ can be replaced by a 'constant' $W(a,b)$. A certain weakness of our result consists in the moment condition 
$\E|X|^{2+\delta}$, which is relevant in the case
$p(c)\le 2$. Greenwood and Perkins impose a milder assumption $\E X^2\log(1+|X|)<\infty$. 

We conclude the introduction by discussing the moment assumption 
$\E|X|^{p(c)}<\infty$. We assume that $c$ is such that $p(c)>2$. 
The following, rather standard, example shows that that moment condition is the minimal one for such values of $c$. Assume that 
\begin{equation}
\label{eq:heavy-tail}
\P(X>t)\ge \theta_0t^{-\beta},\quad t\ge1
\end{equation}
with some $\theta_0>0$ and some $\beta\in(2,p(c))$. Clearly, $\E|X|^{p(c)}$
is then infinite. It is clear that 
$$
\P(T_{a,b}>n)\ge \P(X_1>A\sqrt{n})
\P\left(\min_{k<n}S(k)>\max_{k\le n}g_{a,b}(k)-A\sqrt{n}\right)
$$
for every $A>0$. It follows from \eqref{eq:g-assump.a.b} that 
$$
\frac{1}{\sqrt{n}}\max_{k\le n}g_{a,b}(k)\to c.
$$
Therefore, by the functional central limit theorem,
$$
\P\left(\min_{k<n}S(k)>\max_{k\le n}g_{a,b}(k)-A\sqrt{n}\right)
\to\P\left(\min_{s\le1}B(s)>c-A\right)
$$
and the probability on the right hand side is positive for all $A>c$.
Taking $A=c+1$, we infer that there exists a constant $\theta_1$
such that 
$$
\P(T_{a,b}>n)\ge \theta_1 n^{-\beta/2}.
$$
Thus, the statement of Theorem~\ref{thm:tail} can not hold for walks satisfying \eqref{eq:heavy-tail}.
As we have mentioned before, this extra moment condition is not needed in the case of two-sided boundaries, which can be explained as follows. The second half of the boundary prevents appearance of big jumps as it was described above. Thus, one can think of a random walk with truncated increments.
\section{Properties of the functions \texorpdfstring{$V_p$}{Vp}}\label{sec:Vp} 

We start by collecting some simple but quite useful analytical properties of the functions $V_p(x,t)$. 

The function $ \psi_p(x)=e^{x^2/4}D_p(x)$ is defined for all  real $x$ and $p$.  
For $p<0$ this function admits  the following integral representation 
\begin{equation*}
\psi_p(x) = \frac{1}{\Gamma(-p)} \int\limits_0^{+\infty} \exp\left(-x s - \frac{s^2}{2} \right) s^{-p - 1} ds>0.
\end{equation*}
This implies immediately that $\psi_p$ has no real roots in the case $p<0$. 
It is also known that $\psi_p$ satisfies the following recurrence relation:
\begin{equation}
\label{eq:psi-recur}
\psi_p(x) = x\psi_{p-1}(x) + (1 -p) \psi_{p-2}(x).
\end{equation}
This relation allows one to get an integral representation for $\psi_p$
also in the case $p\ge 0$.

The derivative of $\psi_p$ satisfies
\begin{equation}
\label{eq:psi-prime}
\psi_p'(x)=p\psi_{p-1}(x).
\end{equation}
Recalling definitions \eqref{eq:Vdef.1} and \eqref{eq:Vdef.2} of functions $V_p$ and using \eqref{eq:psi-prime},
we obtain
\begin{equation}
\label{eq:V-deriv}
\frac{\partial}{\partial x}V_p(x,t)=pV_{p-1}(x,t).
\end{equation} 
Combining \eqref{eq:psi-recur} and \eqref{eq:psi-prime}, one gets easily the relation
\begin{equation}
\label{eq:V-2deriv}
\frac{\partial}{\partial t}V_p(x,t)
=-\frac{1}{2}\frac{\partial^2}{\partial x^2}V_p(x,t)
=-\frac{1}{2}p(p-1)V_{p-2}(x,t).
\end{equation}
It is also known that 
\begin{equation}
\label{p0}
\psi_p(x)\sim x^p\quad\text{as }x\to\infty.
\end{equation}
 All  properties presented above may be found in \cite{Atlas}, see also appendices in \cite{Novikov81} and  \cite{U80}.

The following properties of $V_p$ are proven in \cite{DHSW22}.
\begin{lemma}
\label{lem:V-prop}
It holds that
\begin{equation}
\label{p11}
V_p(x,t)\text{ is increasing in }x\in(c\sqrt{t},\infty)
\quad\text{for all}\quad  p\in[0,p(c)]
\end{equation}
and
\begin{align}                                                   \label{p12}
V_{p(c)}(c\sqrt t,t)=0<V_{p(c)}(x,t)\quad\text{for all}\quad 
x>c\sqrt t.
\end{align}
In addition, for all $p\in\R$,
\begin{equation}
\label{p9} 
\sup_{x>z\sqrt{t}}\left|V_p(x,t)/ x^p-1\right|\to0\quad\text{as}\quad z\to\infty. 
\end{equation}
Furthermore, for each $\gamma>0$ there exists a finite positive constant $C_0(\gamma)=C_0(\gamma, c)$ such that
\begin{equation}
\label{p8} 
C_0(\gamma)(x-c\sqrt{t})^{p(c)}\ge V_{p(c)}(x,t)\ge
(x-c\sqrt{t})^{p(c)}/C_0(\gamma) 
\end{equation}
for all $x,t,c$ and $\gamma$ satisfying conditions:
\begin{equation}
\label{p7} 
x-c\sqrt{t}>\gamma\sqrt{t}\quad\text{with}\quad t\ge0. 
\end{equation}
\end{lemma}
Since $c\mapsto p(c)$ is monotone increasing, we can invert this 
mapping. The inverse mapping will be denoted by $p\to c(p)$. It is then clear that $c(p)$ is the largest zero of the function $\psi_p(z)$.

\begin{lemma}
\label{lem:V-boundary}
For every $p>0$ there exists a constant $C=C(p)>0$ such that for $x>c(p)\sqrt{t}$ 
one has 
\begin{equation}\label{eq:derivative.and.v}
C^{-1} {V_{p-1}(x,t)}{(x-c(p)\sqrt{t})}
\le V_{p}(x,t)\le C {V_{p-1}(x,t)}{(x-c(p)\sqrt{t})}.
\end{equation}
If $p>1$ there exists a constant $A(p)$ such that 
$$
\frac{1}{A(p)}x^{p-1}(x-c(p)\sqrt{t})
\le V_p(x,t)\le 
A(p)x^{p-1} (x-c(p)\sqrt{t})
$$
for all $x\ge c(p)\sqrt{t}$.
\end{lemma}
\begin{proof}
Since $\psi_p(x)\sim x^p,x\to \infty$ for all $p$, there exists sufficiently large  $A$ such that \eqref{eq:derivative.and.v} holds for $x>A\sqrt t$. 

Consider now the case $x\in (c(p)\sqrt{t}, A\sqrt{t})$. By the mean value theorem,
\begin{align*}
V_{p}(x,t)
&=V_{p}(x,t)-V_{p}(c(p)\sqrt{t},t)\\
&=\frac{\partial}{\partial \theta} V_{p}(\theta,t)
(x-c(p)\sqrt{t})
= p V_{p-1}(\theta,t) (x-c(p)\sqrt{t}) 
\end{align*}
for some $\theta \in (c(p)\sqrt{t}, x)$.

By the definition of $V_{p-1}$,
$$
V_{p-1}(y,t)=t^{(p-1)/2}\psi_{p-1}(y/\sqrt{t}).
$$
Noting that $1/C(p,A)<\psi_{p-1}(u)<C(p,A)$ for all $u\in[c(p),A]$,
we conclude that the ratio $V_{p-1}(y,t)/V_{p-1}(x,t)$
is bounded below and above by some positive constants, which depend 
on $A$ and $p$ only. This completes the proof of \eqref{eq:derivative.and.v}.

To prove the second claim we notice that $c(p-1)<c(p)$. This allows us to apply \eqref{p8} to the function $V_{p-1}$. 
This gives us the estimates 
$$
\frac{1}{A_1(p)}(x-c(p-1)\sqrt{t})^{p-1}
\le V_{p-1}(x,t)\le
A_1(p)(x-c(p-1)\sqrt{t})^{p-1}
$$
for $x>c(p)\sqrt t$. 
For all $x\ge c(p)\sqrt t $
one has
$$
x-\frac{c(p-1)^+}{c(p)}x\le x-c(p-1)\sqrt{t}\le x+\frac{c(p-1)^-}{c(p)}x,
$$
where as usual $x^+=\max(x,0)$ and $x^-=\max(-x,0)$. 
Using these estimates we obtain 
$$
\frac{1}{A_2(p)}x^{p-1}
\le V_{p-1}(x,t)\le
A_2(p)x^{p-1}.
$$


Thus, the proof is complete.
\end{proof}

\begin{lemma}\label{lem:V-boundary-negative}
	For every $q<0$ and every $T>0$ there exists a constant $A>0$ such that 
	\[V_q(x,t)\leq A(x+T\sqrt{t})^q\]
	for all $x>-T\sqrt{t}$.
\end{lemma}
\begin{proof}
	By the definition,
	\[V_q(x,t)=t^{q/2}\psi_q\Big(\frac{x}{\sqrt{t}}\Big)\]
	and
	\[\psi_q(z)=\frac{1}{\Gamma(-q)}\int_0^\infty \exp\Big(-zs-\frac{s^2}{2}\Big)s^{-q-1}ds.\]
	It is then clear that
	\[\frac{d}{dz}\psi_q(z)=-\frac{1}{\Gamma(-q)}\int_0^\infty\exp\Big(-zs-\frac{s^2}{2}\Big)s^{-q}ds<0.\]
	Thus, $\psi_q(z)$ is decreasing.
	Recall that $\psi_q(z)\sim z^q$ as $z\to\infty$. Therefore, there are constants $R$ and $C_0$ such that
	\[\psi_q(z)\leq C_0\big(\ind_{[-T,R)}(z)+z^q\ind_{[R,\infty)}(z)\big)\]
	and consequently,
	\begin{equation}\label{eq:V-boundary-negative1}
		V_q(x,t)\leq C_0\left(t^{q/2}\ind_{[-T\sqrt{t},R\sqrt{t})}(x)+x^q\ind_{[R\sqrt{t},\infty)}(x)\right).
	\end{equation}
	If $x\leq R\sqrt{t}$ then $(x+T\sqrt{t})^q\geq (R+T)^q\cdot t^{q/2}$. Furthermore,
	for $t>R\sqrt{t}$ we have $(x+T\sqrt{t})^q\geq \left(1+\frac TR\right)^qx^q$. Combining this with \eqref{eq:V-boundary-negative1}, we get the desired estimate.
\end{proof}

\begin{lemma}\label{lem:V-boundary-01}
	For every $p\in(0,1)$ there exists $A(p)$ such that \[V_p(x,t)\leq A(p)(x-c(p)\sqrt{t})^p\]
	for all $x\geq c(p)\sqrt{t}$, $t>0$.
\end{lemma}
\begin{proof}
 Using~\eqref{eq:derivative.and.v} and the fact that 
  $V_{p-1}(x,t)$ is decreasing in $x$ we obtain 
	\[V_p(x,t)\leq A(p) V_{p-1}(c(p)\sqrt{t},t)\cdot(x-c(p)\sqrt{t}).\]
	Applying Lemma \ref{lem:V-boundary-negative} to $V_{p-1}$, we get
	\[V_p(x,t)\leq A_0(p)t^{(p-1)/2}(x-c(p)\sqrt{t})=A_0(p)(x-c(p)\sqrt{t})^p\Big(\frac{x-c(p)\sqrt{t}}{\sqrt{t}}\Big)^{1-p}.\]
 Noting that $\Big(\frac{x-c(p)\sqrt{t}}{\sqrt{t}}\Big)^{1-p}\le 1$
 for $x\in[c(p)\sqrt{t},(c(p)+1)\sqrt{t}]$ we complete the proof.
\end{proof}
\section{Constructions of supermartingales} 
Our constructions of supermartingales for discrete time random walks will use functions $V_p$ which are martingales for the Brownian motion.  Set 
$$
\overline{V}_p(x,t)=\left\{ 
\begin{array}{rl}
V_p(x,t), &x\ge c(p)\sqrt{t},\\
0, &x< c(p)\sqrt{t}.
\end{array}
\right.
$$

In this section we will use an approach which is similar to the method suggested by McConnell~\cite{MC84}, who has constructed supermartingales and derived bounds for exit times from cones by multidimensional walks. 

We first bound the mean drift of the process $\overline{V}_p(a+S(n),b+n)$ for $p>1$ and then for $p\in (0,1)$.
\begin{lemma}[$p>1$]
\label{lem:remainder}

Assume that $\e|X|^{p}<\infty$ and $\e|X|^{2+\delta}<\infty$
for some $\delta>0$.
If $p> 2$ then there exists a constant $C$ such that, for all $x>c(p)\sqrt{t}$,
$$
\left|\e \overline{V}_p(x+X,t+1)-\overline{V}_p(x,t)\right|
\le C\bigg(\frac{(1+x)^{p-1}}{(1+x-c(p)\sqrt{t})^{p-1}}+(1+x)^{p-3}\bigg). 
$$
If $p\in (1,2]$ then there exists a constant $C$ such that, for all $x>c(p)\sqrt{t}$,
$$
\left|\e \overline{V}_p(x+X,t+1)-\overline{V}_p(x,t)\right|
\le C\bigg(\frac{(1+x)^{p-1}}{(1+x-c(p)\sqrt{t})^{1+\delta}}\bigg). 
$$
\end{lemma}
\begin{proof}
Both statements are obvious for $x\in(c(p)\sqrt{t},c(p)\sqrt{t+1}]$. From now on we assume that 
$x>c(p)\sqrt{t+1}$.

First we notice that
\begin{align*}
 &V_p(x+z,t+1)-V_p(x,t)\\
 &\hspace{1cm}=\left[V_p(x+z,t+1)-V_p(x+z,t)\right]
 +\left[V_p(x+z,t)-V_p(x,t)\right].
\end{align*}
Applying now the Taylor formula, we get the equalities
\begin{align*}
&V_p(x+z,t+1)-V_p(x+z,t)\\
&\hspace{1cm}=\frac{\partial}{\partial t}V_p(x+z,t)
+\frac{1}{2}\frac{\partial^2}{\partial t^2}V_p(x+z,t+\phi)\\
&\hspace{1cm}=\frac{\partial}{\partial t}V_p(x,t)
+z\frac{\partial^2}{\partial x\partial t}V_p(x+\psi z,t)
+\frac{1}{2}\frac{\partial^2}{\partial t^2}V_p(x+z,t+\phi)
\end{align*}
for some $\phi,\psi\in(0,1)$.  Similarly,
\begin{align*}
&V_p(x+z,t)-V_p(x,t)\\
&\hspace{1cm}=z\frac{\partial}{\partial x}V_p(x,t)
+z^2\frac{1}{2}\frac{\partial^2}{\partial x^2}V_p(x,t)
+z^3\frac{1}{6}\frac{\partial^3}{\partial x^3}V_p(x+\theta z,t)
\end{align*}
for some $\theta\in(0,1)$.
Using now \eqref{eq:V-deriv} and \eqref{eq:V-2deriv}, we arrive at the equality
\begin{align}
\label{eq:Taylor-repr}
\nonumber
&V_p(x+z,t+1)-V_p(x,t)\\
&=z pV_{p-1}(x,t)+\frac{z^2}{2}p(p-1)V_{p-2}(x,t)
-\frac{1}{2}p(p-1)V_{p-2}(x,t)+R_p(x,t;z),
		\end{align}
where 
\begin{align*}
R_p(x,t;z)
&=\frac{z^3}{6}p(p-1)(p-2)V_{p-3}(x+\theta z,t)\\
&\hspace{5mm}-\frac{z}{2}p(p-1)(p-2)V_{p-3}(x+\psi z,t)\\
&\hspace{5mm}+\frac{1}{8}p(p-1)(p-2)(p-3)V_{p-4}(x,t+\phi), \text{for }  \phi,\psi,\theta\in(0,1). 
\end{align*}
Using the equality \eqref{eq:Taylor-repr}, we have
\begin{align*}
&\hspace{-10mm}\e\left[\overline{V}_p(x+X,t+1)-\overline{V}_p(x,t); |X|\le(x-c(p)\sqrt{t+1})/2\right]\\
&\hspace{-5mm}=\e\left[V_p(x+X,t+1)-V_p(x,t); |X|\le(x-c(p)\sqrt{t+1})/2\right]\\
&\hspace{-5mm}=-\frac{1}{2}p(p-1)V_{p-2}(x,t)\P(|X|\le (x-c(p)\sqrt{t+1})/2)\\
&+pV_{p-1}(x,t)\e[X;|X|\le(x-c(p)\sqrt{t+1})/2]\\
&+\frac{1}{2}p(p-1)V_{p-2}(x,t)\e[X^2;|X|\le(x-c(p)\sqrt{t+1})/2]\\
&+\e[R_p(x,t;X);|X|\le(x-c(p)\sqrt{t+1})/2].
\end{align*}
Recall that $p>1$ is equivalent to $c(p)>0$.  
It follows from \eqref{p8} that 
$V_{p-1}(x,t)=O(x^{p-1})$ and $V_{p-2}(x,t)=O(x^{p-2})$
uniformly in $x>c(p)\sqrt{t}$.

To bound $\e[R_p(x,t;X);|X|\le (x-c(p)\sqrt{t})/2]$, we notice that, by \eqref{p8},
$
V_{p-3}(x+\theta z,t)=O((1+x)^{p-3})
$ and 
$
V_{p-4}(x,t+\phi)=O((1+x)^{p-4})
$
for $x>c(p)\sqrt{t+1}$. 
These estimates imply that 
\begin{align*}
&|\e[R_p(x,t;X);|X|\le (x-c(p)\sqrt{t+1})/2]|\\
&\hspace{1cm}
\le C\left((1+x)^{p-4}+(1+x)^{p-3}\e[|X|^3;|X|\le (x-c(p)\sqrt{t+1})/2]\right).
\end{align*}
Recalling that $\e X=0$ and $\e X^2=1$ and using once again \eqref{eq:V-2deriv}, we infer that 
\begin{align}
\label{eq:repr1}
\nonumber
&\left|\e\left[\overline{V}_p(x+X,t+1)-\overline{V}_p(x,t); |X|\le(x-c(p)\sqrt{t+1})/2\right]\right|\\
\nonumber
&\hspace{5mm}\le C\Big(x^{p-2}\P(|X|>(x-c(p)\sqrt{t+1})/2)\\
\nonumber
&\hspace{1cm}+x^{p-1}\e[X;|X|>(x-c(p)\sqrt{t+1})/2]\\
\nonumber
&\hspace{1cm}+x^{p-2}\e[X^2;|X|>(x-c(p)\sqrt{t+1})/2]\\
&\hspace{1cm}+(1+x)^{p-4}+(1+x)^{p-3}\e[|X|^3;|X|\le (x-c(p)\sqrt{t+1})/2]\Big).
\end{align}

It is immediate from \eqref{eq:V-2deriv} that $\overline{V}_p(x,t)$ increases in $x$
and decreases in $t$. This observation implies that 
$\overline{V}_p(x+X,t+1)\leq \overline{V}_p(x,t+1)\leq \overline{V}_p(x,t)$ for all $X<-(x-c(p)\sqrt{t+1})/2$. 
This implies that 
\begin{align*}
&\left|\e\left[\overline{V}_p(x+X,t+1)-\overline{V}_p(x,t); |X|>(x-c(p)\sqrt{t+1})/2\right]\right|\\
&\hspace{3mm}\le \e\left[\overline{V}_p(x+X,t+1);X>(x-c(p)\sqrt{t+1})/2\right]\\
&\hspace{1cm}
+\overline{V}_p(x,t)\pr(|X|>(x-c(p)\sqrt{t+1})/2).
\end{align*}
Using Lemma~\ref{lem:V-boundary},  
\begin{align}
&\left|\e\left[\overline{V}_p(x+X,t+1)-\overline{V}_p(x,t); |X|>(x-c(p)\sqrt{t+1})/2\right]\right|\notag\\
&\le A(p)\e\left[(x+X-c(p)\sqrt t)\cdot (x+X)^{p-1};X>(x-c(p)\sqrt{t+1})/2\right]\notag\\
&\qquad+A(p)(x-c(p)\sqrt{t})x^{p-1}\pr(|X|>(x-c(p)\sqrt{t+1})/2)\notag\\ 
&\le 3\cdot2^{p-1}A(p)\e\left[|X|(x^{p-1}+|X|^{p-1});X>(x-c(p)\sqrt{t+1})/2\right]\notag\\
&\qquad+2A(p)(x-c(p)\sqrt{t})x^{p-1}\pr(|X|>(x-c(p)\sqrt{t+1})/2)\notag\\
&\le C\Big(\e\left[|X|^p;X>(x-c(p)\sqrt{t+1})/2\right]+
x^{p-1}\e\left[|X|;X>(x-c(p)\sqrt{t+1})/2\right]\notag\\
&\qquad+(x-c(p)\sqrt{t})x^{p-1}\pr(|X|>(x-c(p)\sqrt{t+1})/2)\Big).
\label{eq:repr2}
\end{align}

Assume that $p>2$. 
Due to our moment assumption $\e|X|^p<\infty$,
\begin{align}
	\P(|X|>(x-c(p)\sqrt{t+1})/2)&\le C(1+x-c(p)\sqrt{t})^{-p},\label{eq:markovestim>1}
\end{align}
\begin{align}
	\e[|X|;|X|>(x-c(p)\sqrt{t+1})/2]&\le 
    C(1+x-c(p)\sqrt{t})^{1-p},\label{eq:markovestim>2}
\end{align}
\begin{align} 
	\e[X^2;|X|>(x-c(p)\sqrt{t+1})/2]&\le (1+x-c(p)\sqrt{t})^{2-p},\label{eq:markovestim>3}
 \end{align}
 and
\begin{align}
	\e[|X|^p;X>(x-c(p)\sqrt{t+1})/2]&\le C.\label{eq:markovestim>4}
\end{align}
If $p\ge 3$ then $\e[|X|^3;|X|\le (x-c(p)\sqrt{t})/2]$ is bounded. Plugging in \eqref{eq:markovestim>1} -- \eqref{eq:markovestim>4} into \eqref{eq:repr1} and \eqref{eq:repr2}, we thus obtain
\begin{align*}
	&\left|\e \overline{V}_p(x+X,t+1)-\overline{V}_p(x,t)\right|\\
	&\quad\leq\left|\e\left[\overline{V}_p(x+X,t+1)-\overline{V}_p(x,t); |X|\leq(x-c(p)\sqrt{t+1})/2\right]\right|\\
	&\qquad+\left|\e\left[\overline{V}_p(x+X,t+1)-\overline{V}_p(x,t); |X|>(x-c(p)\sqrt{t+1})/2\right]\right|\\
	&\quad\le C\Big(\frac{x^{p-2}}{(1+x-c(p)\sqrt{t})^{p}}
    +\frac{x^{p-1}}{(1+x-c(p)\sqrt{t})^{p-1}}+
    \frac{x^{p-2}}{(1+x-c(p)\sqrt{t})^{p-2}}\\
	&\hspace{2cm}+(1+x)^{p-4}+(1+x)^{p-3}+1
     +\frac{(x-c(p)\sqrt{t})x^{p-1}}{(1+x-c(p)\sqrt{t})^p}\Big)\\
	&\quad\le  C\left(\frac{(1+x)^{p-1}}{(1+x-c(p)\sqrt{t})^{p-1}}
 +(1+x)^{p-3}\right), 
\end{align*}
which is the desired bound for $p\ge3$.

If $p\in(2,3)$ then 
\begin{align*}
\e[|X|^3;|X|\le (x-c(p)\sqrt{t+1})/2]&\le 
(x-c(p)\sqrt{t+1})^{3-p}\e|X|^p\\
&\le C(1+x-c(p)\sqrt{t})^{3-p}.
\end{align*} 
Applying this together with \eqref{eq:markovestim>1} -- \eqref{eq:markovestim>4} to \eqref{eq:repr1} and \eqref{eq:repr2}, we find for $p\in (2,3)$
\begin{align*}
&\left|\e \overline{V}_p(x+X,t+1)-\overline{V}_p(x,t)\right|\\
&\quad\leq\left|\e\left[\overline{V}_p(x+X,t+1)-\overline{V}_p(x,t); |X|\leq(x-c(p)\sqrt{t+1})/2\right]\right|\\
&\qquad+\left|\e\left[\overline{V}_p(x+X,t+1)-\overline{V}_p(x,t); |X|>(x-c(p)\sqrt{t+1})/2\right]\right|\\
&\quad\le C\frac{x^{p-1}}{(1+x-c(p)\sqrt{t})^{p-1}}.
\end{align*}
This completes the proof of the first claim.

If $1<p\leq 2$ then we assume that $\e|X|^{2+\delta}<\infty$
for some $\delta>0$. In this case, instead of \eqref{eq:markovestim>1} -- \eqref{eq:markovestim>3}, we have
\begin{align}
\pr(|X|>(x-c(p)\sqrt{t+1})/2)
\le C(1+x-c(p)\sqrt{t})^{-(2+\delta)},\label{eq:markovestim<1}
\end{align}
\begin{align}
\e[|X|^{r};X>(x-c(p)\sqrt{t+1})/2]&\le C(1+x-c(p)\sqrt{t})^{r-2-\delta}
\quad\text{for all $r\leq 2$}
\label{eq:markovestim<2}
\end{align}
and
\begin{align}
\nonumber
\e[|X|^{3};X\leq(x-c(p)\sqrt{t+1})/2]
&\leq(x-c(p)\sqrt{t})^{1-\delta}\EW{|X|^{2+\delta}}\\
&\le C(1+x-c(p)\sqrt{t})^{1-\delta}.\label{eq:markovestim<3}
\end{align}
Entering \eqref{eq:markovestim<1} -- \eqref{eq:markovestim<3} into \eqref{eq:repr1} and \eqref{eq:repr2}, we thus obtain for $p\in(1,2]$
\begin{align*}
&\left|\e \overline{V}_p(x+X,t+1)-\overline{V}_p(x,t)\right|\\
&\quad\leq\left|\e\left[\overline{V}_p(x+X,t+1)-\overline{V}_p(x,t); |X|\leq(x-c(p)\sqrt{t+1})/2\right]\right|\\
&\qquad+\left|\e\left[\overline{V}_p(x+X,t+1)-\overline{V}_p(x,t); |X|>(x-c(p)\sqrt{t+1})/2\right]\right|\\
&\quad \le C\frac{(1+x)^{p-1}}{(1+x-c(p)\sqrt{t})^{1+\delta}}, 
\end{align*}
which is the second claim.
\end{proof}

\begin{lemma}[$p\in(0,1)$]\label{lem:remainder01}
	Assume that $\E |X|^{2+\delta}<\infty$ for some $\delta>0$.	
	If $p\in (0,1)$ then, as $x-c(p)\sqrt{t}\to\infty$,
	$$
	\left|\e \overline{V}_p(x+X,t+1)-\overline{V}_p(x,t)\right|
	= o\big({(1+x-c(p)\sqrt{t})^{p-2-\delta}}\big). 
	$$
\end{lemma}

\begin{proof}
	We shall use the same decomposition as in the proof of Lemma \ref{lem:remainder}.  Recalling that $\overline{V}_p(x,t)$ is decreasing in $t$ and increasing in $x$, we obtain
	\begin{align*}
		E_1&:= \left|\E\left[\overline{V}_p(x+X,t+1)-\overline{V}_p(x,t){;}\,X>(x\!-\!c(p)\sqrt{t})/2\right]\right|\\
		&\leq\e\left[\overline{V}_p(x+X,t);X>(x\!-\!c(p)\sqrt{t})/2\right]
		+\overline{V}_p(x,t)\pr(X>(x\!-\!c(p)\sqrt{t})/2)\\
        &\leq2 \e\left[\overline{V}_p(x+X,t);X>(x\!-\!c(p)\sqrt{t})/2\right]\\
		&\leq 2A(p)\e\left[(x\!+\!X\!-\!c(p)\sqrt{t})^p;X>(x\!-\!c(p)\sqrt{t})/2\right]\\
		&\leq 6 A(p)\e\left[|X|^p;X>(x\!-\!c(p)\sqrt{t})/2\right],
	\end{align*}
	where the third inequality follows from Lemma~\ref{lem:V-boundary-01}. Applying now the Chebyshev inequality, we conclude that
	\begin{equation}
		\label{eq:remainder01-1}
		E_1=o\left((1+x-c(p)\sqrt{t})^{p-2-\delta}\right).
	\end{equation}
	On the event $X<-(x\!-\!c(p)\sqrt{t})/2$ we have, due to Lemma~\ref{lem:V-boundary-01}, $\overline{V}_p(x+X,t+1)\le  C(x\!-\!c(p)\sqrt{t})^p$ and $\overline{V}_p(x,t)\le  C(x\!-\!c(p)\sqrt{t})^p$. Therefore,
	\begin{multline}\label{eq:remainder01-2}
		\E \left[\overline{V}_p(x+X,t+1)-\overline{V}_p(x,t){;}\,X<-(x\!-\!c(p)\sqrt{t})/2\right]
        \\=o\left((1+x\!-\!c(p)\sqrt{t})^{p-2-\delta}\right).
	\end{multline}
	In the event $|X|\leq(x-c(p)\sqrt{t})/2$ we shall use the decomposition in \eqref{eq:repr1}. Combining Lemma \ref{lem:V-boundary-negative} with the Chebyshev inequality, we have
	\begin{align*}
		V_{p-2}(x,t)\P(|X|>(x\!-\!c(p)\sqrt{t})/2)
        &\leq C(1+x\!-\!c(p)\sqrt{t})^{p-2}\P(|X|>(x\!-\!c(p)\sqrt{t})/2)\\
		&=o\left((1+x\!-\!c(p)\sqrt{t})^{p-4-\delta}\right),
  \end{align*}
  \begin{align*}
		V_{p-1}(x,t)\E\left[X{;}\,|X|>(x\!-\!c(p)\sqrt{t})/2\right]	=o\left((1+x\!-\!c(p)\sqrt{t})^{p-2-\delta}\right)
  \end{align*}
  and
  \begin{align*}
		V_{p-2}(x,t)\E\left[X^2{;}\,|X|>(x\!-\!c(p)\sqrt{t})/2\right]	&=o\left((1+x\!-\!c(p)\sqrt{t})^{p-2-\delta}\right).
	\end{align*}
	Similarly,
	\[\E\left[R_p(x,t,X){;}\,|X|\leq (x\!-\!c(p)\sqrt{t})/2\right]=o\left((1+x\!-\!c(p)\sqrt{t})^{p-2-\delta}\right).\]
	This completes the proof.
\end{proof}

We have bounded the absolute value of the mean drift. 
To construct a supermartingale 
we will correct $\overline{V}_p$ with a term of order
$x^{p-\delta}$ with some sufficiently small $\delta$.
In the next lemma we estimate the drift of this term.

\begin{lemma}[$p>1$]
	Assume that $\e|X|^{p}<\infty$ and $\e|X|^{2+\delta}<\infty$
	for some $\delta>0$. 
    \begin{enumerate}
    \item If $p>3$, then for all $x>c(p)\sqrt{t}$ there is constant $C>0$ depending only on $p$ and $\delta$ such that
	\begin{multline*}
	\E\left[|x+X|^{p-\delta}-x^{p-\delta}\right]\\
    \geq \frac{(p-\delta)(p-\delta-1)}{2}x^{p-\delta-2}-C\left(\frac{x^{p-\delta}}{(x-c(p)\sqrt{t})^{p}}
	+x^{p-\delta-3}\right).
	\end{multline*}
	\item If $p\in(2,3]$, then for all $x>c(p)\sqrt{t}$ there is a constant $C>0$ depending only on $p$ and $\delta$ such that	
	\begin{align*}		
	\E\left[|x+X|^{p-\delta}-x^{p-\delta}\right]\geq
    & \frac{(p-\delta)(p-\delta-1)}{2}x^{p-\delta-2}\\
    &\hspace{1cm}-C\left(\frac{x^{p-\delta}}{(x-c(p)\sqrt{t})^{p}}
	+(x-c(p)\sqrt t)^{p-2\delta-2}\right).
	\end{align*}
	\item If $p\in(1,2]$, then for all $x>c(p)\sqrt{t}$ there is a constant $C>0$ depending only on $p$ and $\delta$ such that
	\begin{align*}
	\E\left[|x+X|^{p-\delta}-x^{p-\delta}\right]
    &\geq \frac{(p-\delta)(p-\delta-1)}{2}x^{p-\delta-2}\\
    &\hspace{1cm}-C\left(\frac{x^{p-\delta}}{(x-c(p)\sqrt t)^{2+\delta}}
	+(x-c(p)\sqrt t)^{p-2\delta-2}\right).
	\end{align*}
    \end{enumerate}
	\label{lem:correctdiff}
\end{lemma}

\begin{proof}
	We split the expectation into three parts 
	\begin{align}
	\E&\left[|x+X|^{p-\delta}-x^{p-\delta}\right]\notag\\
	&=\E\left[|x+X|^{p-\delta}-x^{p-\delta}{;}\, |X|\leq(x-c(p)\sqrt{t})/2\right]\notag\\
	&\quad+\E\left[|x+X|^{p-\delta}-x^{p-\delta}{;}\, X>(x-c(p)\sqrt{t})/2\right]\notag\\
	&\quad+\E\left[|x+X|^{p-\delta}-x^{p-\delta}{;}\, X<-(x-c(p)\sqrt{t})/2\right]\notag\\
	&=:A_1+A_2+A_3.\label{eq:correct}
	\end{align}
	We find immediately that $A_2\geq 0$. For $A_3$, notice that  $|x+X|^p\geq0$
	and thus due to \eqref{eq:markovestim>1} (if $p>2$) and \eqref{eq:markovestim<1} (if $1<p\leq 2$), there is a constant $C>0$ 
    such that
	\begin{align}
    \label{eq:A3} 
		A_3&\geq -x^{p-\delta}\P(X<-(x-c(p)\sqrt{t})/2)\geq -C\frac{x^{p-\delta}}{(x-c(p)\sqrt{t})^p}.
	\end{align}
	Finally, for $A_1$, notice that due to $X\geq -(x-c(p)\sqrt{t})/2$, $x+X\geq0$ so we can remove the absolutes. By Taylor's formula, there is a $\theta=\theta(x,X)\in[0,1]$ such that
	\begin{align}
	A_1&=\EW{|x+X|^{p-\delta}-x^{p-\delta}{;}\, |X|\leq(x-c(p)\sqrt{t})/2}\notag\\
	\begin{split}
	&=\E\Big[{(p-\delta)}x^{p-\delta-1}\cdot X+\frac{(p-\delta)(p-\delta-1)}{2}x^{p-\delta-2}\cdot X^2\\
	&+\frac{(p-\delta)(p-\delta\!-\!1)(p-\delta\!-\!2)}{6}(x\!+\!\theta X)^{p-\delta-3}\cdot\! X^3{;}\,|X|\leq(x-c(p)\sqrt{t})/2\Big].
	\end{split}\label{eq:correct_B1a}
	\end{align}
	Consider the terms in this expression separately.
	Firstly, due to $\E X=0$ and \eqref{eq:markovestim>2} and \eqref{eq:markovestim<2}, there is a constant $C>0$ 
    such that
	\begin{align}
	\E&\left[(p-\delta)x^{p-\delta-1}\cdot X{;}\,|X|\leq(x-c(p)\sqrt{t})/2\right]\notag\\
	&=-(p-\delta)x^{p-\delta-1}\EW{X{;}\,|X|>(x-c(p)\sqrt{t})/2}\notag\\
    &\geq-C\frac{x^{p-\delta-1}}{(x-c(p)\sqrt{t})^{p-1}}
    \geq-C\frac{x^{p-\delta}}{(x-c(p)\sqrt{t})^{p}}
    \label{eq:correct_B1b}.
	\end{align}
	Secondly, due to $\E X^2=1$ and \eqref{eq:markovestim>3}, there is a constant $C>0$ 
    such that in the case $p>2$
	\begin{align}
	\E&\left[\frac{(p-\delta)(p-\delta-1)}{2}x^{p-\delta-2}\cdot X^2{;}\,|X|\leq (x-c(p)\sqrt{t})/2\right]\notag\\
	&=\frac{(p-\delta)(p-\delta-1)}{2}x^{p-\delta-2}\left(1-\EW{X^2{;}\,|X|>(x-c(p)\sqrt{t})/2}\right)\notag\\
	&\geq\frac{(p-\delta)(p-\delta-1)}{2}x^{p-\delta-2}-C\frac{x^{p-\delta-2}}{(x-c(p)\sqrt{t})^{p-2}}\notag\\
    &\geq\frac{(p-\delta)(p-\delta-1)}{2}x^{p-\delta-2}-C\frac{x^{p-\delta}}{(x-c(p)\sqrt{t})^{p}},\label{eq:correct_B1c}
	\end{align}
	while for $1<p\leq 2$ using \eqref{eq:markovestim<2},  there is a constant $C>0$ 
    such that
	\begin{align}
	\E&\left[\frac{(p-\delta)(p-\delta-1)}{2}x^{p-\delta-2}\cdot X^2{;}\,|X|\leq (x-c(p)\sqrt{t})/2\right]\notag\\
	&=\frac{(p-\delta)(p-\delta-1)}{2}x^{p-\delta-2}\left(1-\EW{X^2{;}\,|X|>(x-c(p)\sqrt{t})/2}\right)\notag\\
	&\geq\frac{(p-\delta)(p-\delta-1)}{2}x^{p-\delta-2}-C\frac{x^{p-\delta-2}}{(x-c(p)\sqrt t)^{\delta}}\notag\\
    &\geq\frac{(p-\delta)(p-\delta-1)}{2}x^{p-\delta-2}-C\frac{x^{p-\delta}}{(x-c(p)\sqrt t)^{2+\delta}}.\label{eq:correct_B1c2}
	\end{align}
	Thirdly, we notice that if $|X|\leq(x-c(p)\sqrt{t})/2$, then for every $\theta\in[0,1]$
	\[\frac32 x\geq \frac32 x-\frac12c(p)\sqrt t \geq x+\theta X\geq \frac12(x+c(p)\sqrt t)\geq 0.\] For $p>3$, the third moment of $X$ is finite, so there is a constant $C>0$ depending only on $p$ and $\delta$ such that
	\begin{align}
	\E&\left[\frac{(p-\delta)(p-\delta-1)(p-\delta-2)}{6}(x+\theta X)^{p-\delta-3}\cdot X^3{;}\,|X|\leq (x-c(p)\sqrt{t})/2\right]\notag\\
	&\geq-\frac{(p-\delta)(p-\delta-1)(p-\delta-2)}{6}\left(\frac32x\right)^{p-\delta-3}\EW{|X|^3{;}\,|X|\leq (x-c(p)\sqrt{t})/2}\notag\\
	&\geq-Cx^{p-\delta-3}\label{eq:correct_B1d}.
	\end{align}
	Combining all results \eqref{eq:correct} -- \eqref{eq:correct_B1c} and \eqref{eq:correct_B1d}, we obtain that for $p>3$ there is a constant $C>0$ depending only on $p$ and $\delta$ such that
	\begin{align*}
	\E\left[|x+X|^{p}-x^{p}\right]
    &\geq \frac{(p-\delta)(p-\delta-1)}{2}x^{p-\delta-2}-C\left(\frac{x^{p-\delta}}{(x-c(p)\sqrt{t})^{p}}
	+x^{p-\delta-3}\right),
	\end{align*}
	If $p\in (2,3]$, then we notice that $p-\delta-3<0$. So if $|X|\leq(x-c(p)\sqrt{t})/2$, then 
	\[0\leq (x+\theta X)^{p-\delta-3}\leq\Big(x-(x-c(p)\sqrt{t})/2\Big)^{p-\delta-3}\leq C\Big(x-c(p)\sqrt{t}\Big)^{p-\delta-3}\]
	and using that $\E |X|^{2+\delta}<\infty$, we obtain
	\begin{align}
	\E&\left[\frac{(p\!-\!\delta)(p\!-\!\delta\!-\!1)(p\!-\!\delta\!-\!2)}{6}(x+\theta X)^{p-\delta-3}\cdot X^3{;}\,|X|\leq (x-c(p)\sqrt{t})/2\right]\notag\\
	&\geq-C\left((x-c(p)\sqrt{t})/2\right)^{p-\delta-3}\EW{|X|^3{;}\,|X|\leq (x-c(p)\sqrt{t})/2}\notag\\
	&\geq-C\left(x-c(p)\sqrt t\right)^{p-\delta-3}\EW{|X|^3\Big(\frac{(x-c(p)\sqrt{t})/2}{|X|}\Big)^{1-\delta}{;}\,|X|\leq (x-c(p)\sqrt{t})/2}\notag\\
	&\geq-C\left(x-c(p)\sqrt t\right)^{p-2\delta-2}\EW{|X|^{2+\delta}{;}\,|X|\leq (x-c(p)\sqrt{t})/2}\notag\\
	&\geq-C\left(x-c(p)\sqrt t\right)^{p-2\delta-2}\label{eq:correct_B1e}.
	\end{align}
	Combining all results \eqref{eq:correct} -- \eqref{eq:correct_B1c} and \eqref{eq:correct_B1e}, for $p\in(2,3]$ we obtain that there is a constant $C>0$ depending only on $p$ and $\delta$ such that
	\begin{align*}
	\E\left[|x+X|^{p}-x^{p}\right]
    &\geq \frac{(p-\delta)(p-\delta-1)}{2}x^{p-\delta-2}\\
    &\hspace{1cm}-C\left(\frac{x^{p-\delta}}{(1+x-c(p)\sqrt{t})^{p}}
	+(x-c(p)\sqrt t)^{p-2\delta-2}\right).
	\end{align*}
	If $p\in (1,2]$, then we notice that $p-\delta-3<p-\delta-2<0$. So analogous to the case $p\in(2,3]$, there is some constant $C>0$ such depending only on $p$ and $\delta$ such that
	\begin{align}
	\E&\left[\frac{(p\!-\!\delta)(p\!-\!\delta\!-\!1)(p\!-\!\delta\!-\!2)}{6}(x+\theta X)^{p-\delta-3}\cdot X^3{;}\,|X|\leq (x-c(p)\sqrt{t})/2\right]\notag\\
	&\geq-C\left(x-c(p)\sqrt t\right)^{p-2\delta-2}\label{eq:correct_B1f}.
	\end{align}
	Combining all results \eqref{eq:correct} -- \eqref{eq:correct_B1b}, \eqref{eq:correct_B1c2} and \eqref{eq:correct_B1f}, for $p\in(1,2]$ we obtain that there is a constant $C>0$ depending only on $p$ and $\delta$ such that
	\begin{align*}
	\E\left[|x+X|^{p}-x^{p}\right]
    &\geq \frac{(p-\delta)(p-\delta-1)}{2}x^{p-\delta-2}\\
    &\hspace{1cm}-C\left(\frac{x^{p-\delta}}{(x-c(p)\sqrt t)^{2+\delta}}
	+(x-c(p)\sqrt t)^{p-2\delta-2}\right),
	\end{align*}
	thereby completing the proof.	
\end{proof}

\begin{remark}
	If $t\to\infty$ and $x\ge(c(p)+\gamma)\sqrt{t}$ 
    for some $\gamma>0$, then we have
	\begin{align*}
	\E&\left[|x+X|^{p-\delta}-x^{p-\delta}\right]=\frac{(p-\delta)(p-\delta-1)}{2}x^{p-\delta-2}(1+o(1)),
	\end{align*}
	as the dominating term was obtained only through equalities.
 \hfill$\diamond$
 \label{rem:asymptexact}
\end{remark}

By Lemma \ref{lem:V-prop}, we notice that if $x\geq (c(p)+\gamma)\sqrt t$ for some $\gamma>0$, then 
\[V_p(x,t)-C\cdot x^{p-\delta}\geq \frac{1}{A(p,\gamma)}x^p-Cx^{p-\delta}.\]
For every constant $C$ we can therefore find a constant $R(C)$ such that
for every $R\ge R(C)$ one has
\begin{equation}
V_p(x+R,t)-C\cdot |x+R|^{p-\delta}>0
\label{eq:Lower_for_f}
\end{equation}
for all $x\geq (c(p)+\gamma)\sqrt t$.

With this in mind, we define for $k\in\N$ and $x\in\R$ 
\[h(x,t):= \left(\overline{V}_p(x+R,t)-C\cdot |x+R|^{p-\delta}\right)^+.\]

\begin{lemma}
	Let $p>1$. For every $\gamma>0$ there exist constants $C$ and $R$ such that if $x\geq (c(p)+\gamma)\sqrt{t}$ then 
	\[\E\left[h(x+X,t+1)-h(x,t)\right]\leq 0.\]
	\label{lem:expectf}
\end{lemma}
\begin{proof}
	There exists a constant $C>0$ depending only on $p$ and $\delta$ such that
	\begin{align}
		&\E\left[|x+R+X|^{p-\delta}{;}\,X<-(x-c(p)\sqrt{t})/2\right]\notag\\
		&\leq 2^{p}\left((x+R)^{p-\delta}\P\left(|X|>(x-c(p)\sqrt{t})/2\right)+\E\left[|X|^{p-\delta}{;}\,|X|>(x-c(p)\sqrt{t})/2\right]\right)\notag\\
		&\leq 2^{2p}\left((x+R)^{p-\delta}\frac{\E[|X|^p]}{(x-c(p)\sqrt t)^p}+\frac{\E\left[|X|^{p}\right]}{(x-c(p)\sqrt t)^\delta}\right)\notag\\
		&\leq C\left(\frac{(x+R)^{p-\delta}}{(x-c(p)\sqrt t)^p}
        +\frac{1}{(x-c(p)\sqrt t)^\delta}\right)\notag\\
        &\leq C\frac{(x+R)^{p-\delta}}{(x-c(p)\sqrt t)^p}.
        \label{eq:expectf-rem1}
	\end{align}
	Using \eqref{eq:Lower_for_f} we infer that
 $$
 \{\overline{V}_p(x+R+X,t+1)<C|x+R+X|^{p-\delta}\}\subseteq\{X<-(x-c(p)\sqrt{t})/2\}
 $$
 for all $R\ge R(C)$, so
	\begin{align*}
	&\E\left[h(x+X,t+1)-h(x,t)\right]\\
	&=\E\left[ \overline{V}_p(x+R+X,t+1)-\overline{V}_p(x+R,t)-C\big(|x+R+X|^{p-\delta}-|x+R|^{p-\delta}\big)\right]\\
	&\quad-\E\Big[ \overline{V}_p(x+R+X,t+1)
    \\&\quad \quad \quad 
    -C|x+R+X|^{p-\delta}{;}\,\overline{V}_p(x+R+X,t+1)
     <C|x+R+X|^{p-\delta}\Big]\\
	&\leq \EW{\overline{V}_p(x+R+X,t+1)-\overline{V}_p(x+R,t)}
    \\
    &\quad -C\cdot\EW{ \left(|x+R+X|^{p-\delta}-|x+R|^{p-\delta}\right)}\\
	&\quad+\E\left[C|x+R+X|^{p-\delta}{;}\,X<-(x-c(p)\sqrt{t})/2\right].
	\end{align*}
	Using  Lemmas \ref{lem:remainder} and \ref{lem:correctdiff} together with \eqref{eq:expectf-rem1}, we find that for $p> 3$,
	\begin{align*}
	\E\left[h(x+X,t+1)-h(x,t)\right]&\leq-\frac{(p-\delta)(p-\delta-1)}{2}(x+R)^{p-\delta-2}\\
	&\quad+C\left((x+R)^{p-3}+ \frac{(x+R)^{p-\delta}}{(x-c(p)\sqrt t)^p}\right).
	\end{align*}
	For $p\in(2,3]$,
	\begin{align*}
	\E\left[h(x+X,t+1)-h(x,t)\right]&\leq-\frac{(p-\delta)(p-\delta-1)}{2}(x+R)^{p-\delta-2}\\
	&+C\bigg(\frac{(x+R)^{p-1}}{(x-c(p)\sqrt{t})^{p-1}}
 +(x+R-c(p)\sqrt t)^{p-2\delta-2}\bigg).
	\end{align*}
	And if $p\le2$, 
	\begin{align*}
	\E\left[h(x+X,t+1)-h(x,t)\right]
    &\leq- \frac{(p-\delta)(p-\delta-1)}{2}(x+R)^{p-\delta-2}\\
    &\quad+C\Bigg(\frac{(x+R)^{p-\delta}}{(x-c(p)\sqrt{t})^{2+\delta}}
	+(x+R-c(p)\sqrt t)^{p-2\delta-2}\Bigg).
	\end{align*}
	
	In all three cases, the entire term is dominated by $-\frac{(p-\delta)(p-\delta-1)}{2}(x+R)^{p-\delta-2}$ for sufficiently large $x$ (all further terms are of lower order; indeed by Remark \ref{rem:asymptexact}, we obtain an asymptotic equality). Due to the negative sign, this term is non-positive, hence
	\[
	\E\left[h(x+X,t+1)-h(x,t)\right]\leq 0.\qedhere\]	
\end{proof}

\begin{lemma}
	For every $c>0$ and every $1<p_1<p(c)$, there exist constants $C$ and $R$ such that 
	\begin{align*}
	M_n&:= h(a+S(T_{a,b}\wedge n),b+T_{a,b}\wedge n)\\
	&=\Big(\overline{V}_{p_1}(R+a+S(T_{a,b}\wedge n),b+T_{a,b}\wedge n)-C\cdot|a+R+S(n)|^{p_1-\delta}\Big)^+
	\end{align*} is a supermartingale.
	\label{lem:supermartingale}
\end{lemma}

\begin{proof}
	Consider
	\begin{align*}
	&\EW{M_{n+1}-M_n\mid \mathcal{F}_n}\\
    &\hspace{1cm}=\EW{(M_{n+1}-M_n)\ind_{\{T_{a,b}\leq n\}}\mid \mathcal{F}_n}+\EW{(M_{n+1}-M_n)\ind_{\{T_{a,b}> n\}}\mid \mathcal{F}_n}.
	\end{align*}
	The first term obviously satisfies
	\begin{align*}
	\E&\left[{(M_{n+1}-M_n)\ind_{\{T_{a,b}\leq n\}}\mid \mathcal{F}_n}\right]\\
	&=\EW{\left(f(a+S(T_{a,b}),T_{a,b}+b)-f(a+S(T_{a,b}),T_{a,b}+b)\right)\ind_{T_{a,b}\leq n}\mid \mathcal{F}_n}=0.
	\end{align*}
	For the other term, we obtain from Lemma \ref{lem:expectf}
    with $\gamma=c(p)-c(p_1)>0$ 
	\begin{align*}
	\E&\left[{(M_{n+1}-M_n)\ind_{\{T_{a,b}>n\}}\mid \mathcal{F}_n}\right]\\
	&=\E\left[{h(a+S(n+1),n+1+b)-h(a+S(n),n+b)\mid \mathcal{F}_n}\right]\ind_{\{T_{a,b}>n\}}\\
	&=\E\left[{h(a+S(n)+X,n+1+b)-h(a+S(n),n+b)\mid \mathcal{F}_n}\right]\ind_{\{T_{a,b}>n\}}\leq 0.
	\end{align*}
	This proves the claim.
\end{proof}

We now turn to the case $p<1$. It turns out that the function $h(x,t)$ is not appropriate for this case and that one has to use a different type of correction for $V_p(x,t)$. We define
$$
G(x,t):=\overline{V}_p(x+R,t+R)-C(t+R)^{(p-\delta)/2}
\quad\text{and}\quad 
g(t,x):=(G(x,t))^+.
$$
If $x>(c(p)+\gamma)\sqrt{t}$ for some $\gamma>0$ then
$$
\overline{V}_p(x+R,t+R)=(t+R)^{p/2}\psi_p\left(\frac{R+x}{\sqrt{R+t}}\right)
\ge (t+R)^{p/2}\psi_p(c(p)+\gamma)
$$
for all sufficiently large $R$. Therefore, there exists $R(C)$ such that, for every $R\ge R(C)$,
\begin{equation}
\label{Lower_for_g}
g(x,t)>0 \quad\text{for all }x\ge(c(p)+\gamma)\sqrt{t}.
\end{equation}
\begin{lemma}
\label{lem:drfit<1}
Assume that $p<1$. For every $\gamma>0$ there exist constants $C$ and $R$
such that, for all $t\ge0$ and $x\ge (c(p)+\gamma)\sqrt{t}$, 
$$
\E[g(x+X,t+1)-g(x,t)]\le0.
$$
\end{lemma}
\begin{proof}
It follows from \eqref{Lower_for_g} that 
\begin{align*}
\E[g(x+X,t+1)-g(x,t)]
&=\E\left[\overline{V}_p(x+R+X,t+R+1)-\overline{V}_p(x+R,t+R)\right]\\
&\hspace{1cm}-C\left[(t+R+1)^{(p-\delta)/2}-(t+R)^{(p-\delta)/2}\right]\\
&\hspace{1cm}-\E\left[G(x+X,t+1);G(x+X,t+1)<0\right].
\end{align*}
For every $x>(c(p)+\gamma)\sqrt{t}$ we have by Lemma \ref{lem:remainder01}
\begin{align*}
&\left|\E\left[\overline{V}_{p}(R+x+X,R+t+1)\right]-\overline{V}_{p}(R+x,R+t)\right|\\
&\hspace{1cm}\leq C_1(1+R+x-c(p)\sqrt{t+R})^{p-2-\delta}
\le C_2 (t+R)^{p/2-\delta/2-1}.
\end{align*}
Furthermore,
\begin{align*}
\left[(R+t+1)^{(p-\delta)/2}-(R+t)^{(p-\delta)/2}\right]
\geq \frac{p-\delta}{2}(R+t+1)^{(p-\delta)/2-1}.
\end{align*}
Finally, for all sufficiently large $R$, $\{G(x+X,t+1)<0\}\subset\{X<-\frac{\gamma}{2}\sqrt{R+t}\}$ and, consequently,
\begin{align*}
&-\E\left[G(x+X,t+1);G(x+X,t+1)<0\right]\\
&\hspace{1cm}\le C(t+R)^{(p-\delta)/2}\P(G(x+X,t+1)<0)\\
&\hspace{1cm}\le C(t+R)^{(p-\delta)/2}\P(X<-\frac{\gamma}{2}\sqrt{R+t})\\
&\hspace{1cm}\le C(t+R)^{(p-\delta)/2-1-\delta/2}\E|X|^{2+\delta}.
\end{align*}
Combining these estimates and increasing if necessary $R$, we get the desired property.
\end{proof}
\begin{lemma}\label{lem:supermartingale<}
	If $p_1<p<1$ then there exist constants $C$ and $R$ such that
	\begin{multline*}
    \widehat M_n:=\Big(V_{p_1}(R+a+S(T_{a,b}\wedge n),R+b+T_{a,b}\wedge n)
    -C(R+b+T_{a,b}\wedge n)^{(p_1-\delta)/2}\Big)^+
    \end{multline*}
	is a supermartingale.
\end{lemma}
The proof of this lemma repeats the proof of Lemma~\ref{lem:supermartingale}
and we omit it.

\section{Upper bounds for \texorpdfstring{$\P(T_{a,b}>n)$}{tabn} and for \texorpdfstring{$\P(S(n)\in(x,x+h],T_{a,b}>n)$}{tabn}}
\subsection{Estimates for the tail of \texorpdfstring{$T_{a,b}$}{t.a.b}.}
\begin{lemma}\label{lem:optionalstopping}
	For every $c>0$ and every $1<p_1<p(c)$, there exist constants $C$ and $R$ such that 
	\[\E\left[\overline{V}_{p_1}(R+a+S(n),n+b){;}\,T_{a,b}>n\right]\leq C (1+a^{p_1}).\]	
	For every $c<0$ and every $0<p_1<p(c)<1$, there exist constants $C$ and $R$ such that 
	\[\E\left[\overline{V}_{p_1}(R+a+S(n),n+b){;}\,T_{a,b}>n\right]\leq C(1+|a|^{p_1}+b^{p_1/2})+Cn^{(p_1-\delta)/2}\P(T_{a,b}>n).\]	
\end{lemma}

\begin{proof}
Consider first $p(c)>1$ and let $p_1$ be such that $1<p_1<p(c)$.  Notice that $T_{a,b}\wedge n$ is a bounded stopping time. 
Hence, we can apply the optional stopping theorem to the supermartingale $M_n$ from Lemma \ref{lem:supermartingale} to obtain that there is a constant $C'$ such that
	\begin{align}
	\E&\left[h(a+S(n),b+n){;}\,T_{a,b}>n\right]\notag
	\\&\leq \E\left[h(a+S(n),b+n){;}\,T_{a,b}>n\right]+\E\left[h(a+S(T_{a,b}),b+T_{a,b}){;}\,T_{a,b}\leq n\right]\notag\\
	&=\E\left[h(a+S(T_{a,b}\wedge n),b+T_{a,b}\wedge n)\right]=\EW{M_n}\notag\\
	&\leq \EW{M_0}= \left(\overline{V}_{p_1}(a+R,b)-C\cdot (a+R)^{p_1-\delta}\right)^+\notag\\
	&\leq(C'(a+R)^{p_1}-C\cdot (a+R)^{p_1-\delta})\leq C'(a+R)^{p_1}. 
	\label{eq:optionalstopping}
	\end{align}
	Notice that on the event $\{T_{a,b}>n\}$, for some sufficiently small $\gamma_0>0$
	\begin{align*}
		c(p_1)\sqrt{n+b}<(c(p)-\gamma_0)\sqrt{n+b}&\leq \frac{c(p)-\gamma_0}{c(p)}c(p)\sqrt{n+b}\\
		&<(1-\tfrac{\gamma_0}{c(p)})(a+S(n))
	\end{align*}
	so in other words, for $\gamma=\frac{c(p_1)\gamma_0}{c(p)}\frac{1}{1-\frac{\gamma_0}{c(p)}}>0$
	\[(a+S(n))>(c(p_1)+\gamma)\sqrt{n+b}\]
	and $a+S(n)-c(p_1)\sqrt{n+b}>\gamma\sqrt{n+b}$.
	It follows from Lemma \ref{lem:V-boundary} that
	\begin{align}
		&\overline{V}_{p_1}(R\!+\!a\!+\!S(n),b\!+\!n)-C(R\!+\!a\!+\!S(n))^{p_1-\delta}\notag\\
		&\geq \frac{1}{A(p_1)}(R\!+\!a\!+\!S(n)-c(p_1)\sqrt{n\!+\!b})\cdot (R\!+\!a\!+\!S(n))^{p_1-1}-C(R\!+\!a\!+\!S(n))^{p_1-\delta}\notag\\
		&=\frac{(R\!+\!a\!+\!S(n))^{p_1-1}}{A(p_1)}\left[R\!+\!a\!+\!S(n)-c(p_1)\sqrt{n\!+\!b}-CA(p_1)(R\!+\!a\!+\!S(n))^{1-\delta}\right]\notag\\
		&\geq\frac{\overline{V}_{p_1}(R\!+\!a\!+\!S(n),b\!+\!n)}{A^2(p_1)}\left[1-CA(p_1)\frac{(R\!+\!a\!+\!S(n))^{1-\delta}}{R\!+\!a\!+\!S(n)-c(p_1)\sqrt{n\!+\!b}}\right].\label{eq:boundexpvp-1}
	\end{align}
	Using the fact that 
	\begin{align*}
		&R\!+\!a\!+\!S(n)-c(p_1)\sqrt{n\!+\!b}\\
        &=R+\left(1-\frac{c(p_1)}{c}\right)(a\!+\!S(n))+\frac{c(p_1)}{c}(a\!+\!S(n)-c\sqrt{n\!+\!b})\\
		&\geq(1-\tfrac{c(p_1)}{c})(R\!+\!a\!+\!S(n))
	\end{align*}
	we obtain from \eqref{eq:boundexpvp-1}
	\begin{align*}
    \overline{V}_{p_1}&(R\!+\!a\!+\!S(n),b\!+\!n)-C(R\!+\!a\!+\!S(n))^{p_1-\delta}\\
    &\geq \frac{\overline{V}_{p_1}(R\!+\!a\!+\!S(n),b\!+\!n)}{A^2(p_1)}\left(1-\frac{CA(p_1)}{R^\delta(1-\frac{c(p_1)}{c})}\right).
	\end{align*}
	Let $K:=\frac{1}{A^2(p_1)}\left(1-\frac{CA(p_1)}{R^\delta(1-\frac{c(p_1)}{c})}\right)$, which is positive for sufficiently large $R$. Then
	\begin{align}
	\E&\left[\overline{V}_{p_1}(R+a+S(n),n+b){;}\,T_{a,b}>n\right]\notag\\
	&\leq\frac1K\E\left[{\overline{V}_{p_1}(R+a+S(n),n+b)-C|R+a+S(n)|^{p_1-\delta}{;}\,T_{a,b}>n}\right]\notag\\
	&\leq\frac1K\E\left[{\left(\overline{V}_{p_1}(R+a+S(n),n+b)-C\cdot |R+a+S(n)|^{p_1-\delta}\right)^+{;}\,T_{a,b}>n}\right]\notag\\
	&=\frac1K\E\left[h(a+S(n),n+b){;}\,T_{a,b}>n\right]\leq \frac CK(a+R)^{p_1}\notag
	\end{align}
	by \eqref{eq:optionalstopping} and the claim follows.
	
	For $0<p_1<p(c)<1$, we can apply optional stopping theorem to the supermartingale $\widehat M_n$ from Lemma \ref{lem:supermartingale<} to obtain 
	\begin{align}
	\E&\left[\overline{V}_{p_1}(R+a+S(n),R+b+n){;}\,T_{a,b}>n\right]\notag\\
	&=	\E\left[\overline{V}_{p_1}(R+a+S(T_{a,b}\wedge n),R+b+T_{a,b}\wedge n){;}\,T_{a,b}>n\right]\notag\\
	&=	\E\left[\widehat M_n{;}\,T_{a,b}>n\right]+C(R+b+n)^{(p_1-\delta)/2}\P(T_{a,b}>n)\notag\\
	&\leq	\E\left[\widehat M_n\right]+C(R+b+n)^{(p_1-\delta)/2}\P(T_{a,b}>n)\notag\\
	&\leq \E\left[\widehat M_0\right]+C(R+b+n)^{(p_1-\delta)/2}\P(T_{a,b}>n)\notag\\
	&\leq \overline{V}_{p_1}(R+a,R+b)+Cn^{(p_1-\delta)/2}\P(T_{a,b}>n).\notag
	\end{align} 
	Using Lemma \ref{lem:V-boundary-01}, we can bound the first term 
	$\overline{V}_{p_1}(R+a,R+b)\leq A(p_1)(R+|a|+\sqrt{b})^{p_1}$,
	so the claim holds.
\end{proof}

\begin{lemma}\label{lem:probprodbound}
	For $x\geq 0$, $z\geq n$, $R\in\R_+$ and $p>0$ we have
	\[\P(\overline{V}_p(S(n)+R,z)>x,T_{a,b}>n)\geq \P(\overline{V}_p(S(n)+R,z)>x)\cdot\P(T_{a,b}>n).\]
\end{lemma}
\begin{proof}
	Since $\overline{V}_p(x,y)$ is increasing in $x$, 
	\[\{\overline{V}_p(a+R+S(n),b+n)>x\}=\{S(n)>g_n\},\]
	where $g_n$ is the unique solution to $\overline{V}_p(a+R+g_n,b+n)=x$. By Lemma 24 in \cite{DSW16},
	\[\P(S(n)>g_n,T_{a,b}>n)\geq \P(S(n)>g_n)\P(T_{a,b}>n).\]
	This gives the desired estimate.
\end{proof}

\begin{lemma}\label{lem:probestim}
	If $\E |X|^p<\infty$, then
	\[\P(T_{a,b}>n)\leq \frac{\EW{\overline{V}_p(a+R+S(n),b+n){;}\, T_{a,b}>n}}{\EW{\overline{V}_p(a+R+S(n),b+n)}}.\]
\end{lemma}
\begin{proof}
	By Lemma \ref{lem:probprodbound},
	\begin{align*}
		&\E\left[\overline{V}_p(a\!+\!R\!+\!S(n),b\!+\!n){;}\, T_{a,b}>n\right]\\
        &\hspace{1cm}=\int_0^\infty\! \P\left(\overline{V}_p(a\!+\!R\!+\!S(n),b\!+\!n)>x, T_{a,b}>n\right)dx\\
		&\hspace{1cm}\geq \P(T_{a,b}>n)\int_0^\infty\!\P\left(\overline{V}_p(a\!+\!R\!+\!S(n),b\!+\!n)>x\right)dx\\
		&\hspace{1cm}=\P(T_{a,b}>n)\E\left[\overline{V}_p(a\!+\!R\!+\!S(n),b\!+\!n)\right].\qedhere
	\end{align*}
\end{proof}

\begin{lemma}\label{lem:Tbound}
	For every $1<p_1<p(c)$, there exists a finite constant $C$ such that for all $a,b$,
	\[\P({T_{a,b}>n})\leq C\frac{1+|a|^{p_1}}{n^{p_1/2}}.\] 
    If $0<p_1<p(c)<1$ then 
    \[\P({T_{a,b}>n})\leq C\frac{1+|a|^{p_1}+b^{p_1/2}}{n^{p_1/2}}.\]
\end{lemma}
\begin{proof}
    Assume first that $p(c)>1$.
	We know from Lemma \ref{lem:optionalstopping} that 
	\[\E\left[\overline{V}_{p_1}(a\!+\!R\!+\!S(n),b\!+\!n){;}\,T_{a,b}>n\right]\leq C (1+a^{p_1})\]
	for all $a$ and $b$. Thus, by Lemma \ref{lem:probestim}, it remains to estimate $\E\left[\overline{V}_{p_1}(a\!+\!R\!+\!S(n),b\!+\!n)\right]$. 
    By the definition of $\overline{V}_{p_1}$, 
	\begin{align*}
	&\E\left[\overline{V}_{p_1}(a\!+\!R\!+\!S(n),b\!+\!n)\right]\\\
    &\hspace{1cm}
    \geq (b+n)^{p_1/2}\E\left[\psi_{p_1}\left(\frac{a\!+\!R\!+\!S(n)}{\sqrt{n+b}}\right){;}\,a+S(n)>c\sqrt{b+n}\right].
    \end{align*}
	Applying the central limit theorem, we conclude that the expectation on the right hand side converges to a positive constant, which depends only on $p$. This gives the desired bound.
 
	In the case  $p_1<p(c)<1$ we use the same estimation for the denominator in Lemma \ref{lem:probestim}. For the enumerator, we find by Lemma \ref{lem:optionalstopping} that there are constants $C_1$ and $C_2$ such that
	\begin{align*}
	    &\E\left[\overline{V}_{p_1}(R+a+S(n),n+b){;}\,T_{a,b}>n\right]\\
        &\hspace{1cm}\leq C_1(1+|a|^{p_1}+b^{p_1/2})+C_2n^{(p_1-\delta)/2}\P(T_{a,b}>n).
    \end{align*}
	Overall, we thus obtain
	\[P(T_{a,b}>n)\leq C_1\frac{1+|a|^{p_1}+b^{p_1/2}}{n^{p_1/2}}+C_2\frac{n^{(p_1-\delta)/2}\P(T_{a,b}>n)}{n^{p_1/2}}.\]
	Rearranging gives some constant $C_2$ such that
	\[\P(T_{a,b}>n)\leq\left(1-\frac{C_2}{n^{\delta/2}}\right)^{-1} C_1\frac{1+|a|^{p_1}+b^{p_1/2}}{n^{p_1/2}}.\] 
    For $n\ge(2C_2)^{2/\delta}$
    we have the desired inequality with $C=2C_1$. This completes the proof of the lemma.
\end{proof}

\begin{lemma}\label{lem:BDG}
	For every $0<p_1<p(c)$, there exists a constant $C<\infty$ such that
	\[\E\left[\left(\max\limits_{k\leq T_{a,b}}|S(k)|\right)^{p_1}\right]
    \leq C(1+|a|^{p_1}+b^{p_1/2}\ind\{p(c)<1\})\]
	and 
	$$
	\E(T_{a,b})^{p_1/2}\leq C(1+|a|^{p_1}+b^{p_1/2}\ind\{p(c)<1\}).
	$$
\end{lemma}
\begin{proof}
	By the Burkholder-Davis-Gundy inequality,
	\[\E\left[\left(\max\limits_{k\leq T_{a,b}}|S(k)|\right)^{p_1}\right]\leq C\E\left[T_{a,b}^{p_1/2}\right].\]
	For every $k\geq 0$, one has
	\begin{align*}
		\E\left[T_{a,b}^{p_1/2}\right]&=\frac{p_1}{2}\int_0^\infty u^{p_1/2-1}\P(T_{a,b}>u)du\\
		&\leq k^{p_1/2}+\frac{p_1}{2}\int_k^\infty u^{p_1/2-1}\P(T_{a,b}>u)du.
	\end{align*}
	Applying Lemma \ref{lem:Tbound} with some $p_2\in(p_1,p)$, we get
	\begin{align*}
		\frac{p_1}{2}&\int_k^\infty u^{p_1/2-1}\P(T_{a,b}>u)du\leq \frac{p_1}{2}\int_k^\infty u^{p_1/2-1}\frac{1+a^{p_2}}{u^{p_2/2}}du\\
		&=\frac{p_1}{2}(1+a^{p_2})\int_k^\infty {u^{(p_1-p_2)/2-1}}du\leq \widehat C(1+a^{p_2})k^{(p_1-p_2)/2}
	\end{align*}
	Taking $k=(1+a)^2$, we get
	\[	\E\left[T_{a,b}^{p_1/2}\right]\leq (1+a)^{p_1}\left(1+\widehat C\frac{1+a^{p_2}}{(1+a)^{p_2}}\right)\leq (1+\widehat C)(1+a)^{p_1}.\]
	This finishes the proof.	
\end{proof}
\subsection{Estimates for local probabilities.}
\begin{lemma}
\label{lem:loc.prob1}
For every $p_1<p(c)$ there exists a constant $C$ such that 
$$
\P(S(n)\in(x,x+1],T_{a,b}>n)\le C\frac{1+|a|^{p_1}}{n^{(p_1+1)/2}}.
$$
\end{lemma}
\begin{proof}
Set $m=\lfloor\frac{n}{2}\rfloor$.
It follows from the local central limit theorem that 
$$
\P(S(n-m)\in(z,z+1])\le C_1(n-m)^{-1/2}\le C_2n^{-1/2} 
$$
uniformly in $z$. Using this bound we obtain 
\begin{align*}
&\P(S(n)\in(x,x+1],T_{a,b}>n)\\
&\le \int \P(S(m)\in dz,T_{a,b}>m)\P(S(n-m)\in(x-z,x-z+1])\\
&\le C_2n^{-1/2}\P(T_{a,b}>m)).
\end{align*}
Applying now Lemma~\ref{lem:Tbound}, we get the desired inequality.
\end{proof}
\begin{lemma}
\label{lem:loc.prob2}
If $c>0$ then for every $1<p_1<p(c)$ there exists a constant $C$ such that, for all $y\le \sqrt{n}/2$,
\begin{equation}
\label{eq:lp1}
\P(a+S(n)\in(c\sqrt{b+n}+y,c\sqrt{b+n}+y+1],T_{a,b}>n)
\le C(1+y)^{1/2}\frac{1+|a|^{p_1}}{n^{p_1/2+3/4}}.
\end{equation}
If $c<0$ then for every $0<p_1<p(c)$ there exists a constant $C$ such that,
for all $y\le \sqrt{n}/2$,
\begin{equation}
\label{eq:lp2}
\P(a+S(n)\in(c\sqrt{b+n}+y,c\sqrt{b+n}+y+1],T_{a,b}>n)
\le C(1+y)\frac{1+|a|^{p_1}+b^{p_1/2}}{n^{p_1/2+1}}.
\end{equation}
\end{lemma}
\begin{proof}
Set 
$$
\Delta(y)=(c\sqrt{b+n}+y,c\sqrt{b+n}+y+1].
$$
Then, by the Markov property at time $n-m<n$,
\begin{align*}
&\P(a+S(n)\in\Delta(y),T_{a,b}>n)\\
&=\int_{c\sqrt{b+n-m}}^\infty \P(a+S(n-m)\in dz, T_{a,b}>n-m)
\P(z+S(m)\in \Delta(y),T_{z,b+n-m}>m).
\end{align*}
Using the duality lemma for random walks, we get 
\begin{align*}
&\P(z+S(m)\in \Delta(y),T_{z,b+n-m}>m)\\
&\le\P\left(z+S(m)\in \Delta(y),\min_{j\le m}(z+S(j))>\min_{j\le m}c\sqrt{b+n-m+j}\right)\\
&\le \P\left(y+1-S(m)\in[z-c\sqrt{b+n},z+1-c\sqrt{b+n}),\right.\\
&\hspace{3cm}\left.\min_{j\le m}(y+1-S(j))>-c\sqrt{b+n}+\min_{j\le m}c\sqrt{b+n-m+j}\right).
\end{align*}
It follows from Lemma~\ref{lem:loc.prob1} that, for $m\le n/2$,
\begin{align*}
&\P(a+S(n-m)\in(c\sqrt{b+n-m}+k,c\sqrt{b+n-m}+k+1],T_{a,b}>n-m)\\
&\hspace{1cm}\le C\frac{1+|a_1|^{p_1}}{n^{(p_1+1)/2}}
\end{align*}
uniformly in $k\ge0$. Therefore, 
\begin{align*}
&\P(a+S(n)\in\Delta(y),T_{a,b}>n)\\
&\le C\frac{1+|a_1|^{p_1}}{n^{(p_1+1)/2}}
\sum_{k=0}^\infty\P\Big(y+1-S(m)\in \Delta'(k),\\
&\hspace{3cm}\min_{j\le m}(y+1-S(j))>-c\sqrt{b+n}+\min_{j\le m}c\sqrt{b+n-m+j}\Big)\\
&\le C\frac{1+|a_1|^{p_1}}{n^{(p_1+1)/2}}
\P\left(\min_{j\le m}(y+1-S(j))>-c\sqrt{b+n}+\min_{j\le m}c\sqrt{b+n-m+j}\right),
\end{align*}
where 
$$
\Delta'(k)=[c\sqrt{b+n-m}-c\sqrt{b+n}+k,c\sqrt{b+n-m}-c\sqrt{b+n}+k+1].
$$
We next notice that 
$$
c\sqrt{b+n}-\min_{j\le m}c\sqrt{b+n-m+j}
=c^+(\sqrt{b+n}-\sqrt{b+n-m}).
$$

Therefore,
\begin{align}
\nonumber
\label{eq:lp3}
&\P(a+S(n)\in\Delta(y),T_{a,b}>n)\\
&\le C\frac{1+|a_1|^{p_1}}{n^{(p_1+1)/2}}
\P\left(\min_{j\le m}(y+1-S(j))>-c^+(\sqrt{b+n}-\sqrt{b+n-m})\right).
\end{align}
If $c>0$ ($c^+=c$) then we take $m=\lfloor (y+1)\sqrt{n}+1\rfloor$.
In this case 
$$
\sqrt{b+n}-\sqrt{b+n-m}=\frac{m}{\sqrt{b+n}+\sqrt{b+n-m}}
\ge\frac{(y+1)\sqrt{n}}{\sqrt{n}}=y+1.
$$
Then, using Lemma 3 from \cite{DW16}, we get 
\begin{align}
\label{eq:lp4}
\nonumber
&\P(\min_{j\le m}(y+1-S(j))>-c^+(\sqrt{b+n}-\sqrt{b+n-m})\\
&\le \P(\min_{j\le m}(-S(j))>-(1+c)(1+y))
\le C\frac{1+y}{\sqrt{m}}\le C\frac{(1+y)^{1/2}}{n^{1/4}}.
\end{align}
Combining \eqref{eq:lp3} and \eqref{eq:lp4}, we get \eqref{eq:lp1}.

If $c^+=0$, i.e. $c\le0$, then we take $m=\lfloor n/2\rfloor$. Using once again Lemma 3 from \cite{DW16}, we have 
\begin{align}
\label{eq:lp5}
\nonumber
&\P(\min_{j\le m}(y+1-S(j))>-c^+(\sqrt{b+n}-\sqrt{b+n-m})\\
&\le \P(\min_{j\le m}(-S(j))>-(1+y))
\le C\frac{1+y}{\sqrt{m}}\le C\frac{(1+y)}{n^{1/2}}.
\end{align}
Combining \eqref{eq:lp3} and \eqref{eq:lp5}, we get \eqref{eq:lp2}.
\end{proof}
\section{Construction of the harmonic function: proof of Theorem~\ref{thm:space-time-h}}
Now let 
$$
f(x,y):=\EW{\overline{V}_{p(c)}(x+X,y+1)}-\overline{V}_{p(c)}(x,y).
$$
Then the sequence
\begin{align*}
M_n:=\overline{V}_{p(c)}(a+S(n),b+n)-\sum\limits_{k=0}^{n-1}f(a+S(k),b+k)    
\end{align*}
is a martingale. 
Therefore, using optional stopping, we obtain
\begin{align*}
	\overline{V}_{p(c)}(a,b)&=M_0=\EW{M_{T_{a,b}\wedge n}}=\EW{M_n{;}\,{T_{a,b}> n}}+\EW{M_{T_{a,b}}{;}\,{T_{a,b}\leq  n}}\\
	&=\EW{\overline{V}_{p(c)}(a+S(n),b+n){;}\,{T_{a,b}> n}}+\EW{\sum\limits_{k=0}^{n-1}f_k{;}\,{T_{a,b}> n}}\\
	&\qquad+\EW{\overline{V}_{p(c)}(a+S(T_{a,b}),b+T_{a,b}){;}\,{T_{a,b}\leq n}}-\EW{\sum\limits_{k=0}^{T_{a,b}-1}f_k{;}\,{T_{a,b}\leq n}},
\end{align*}
where  $f_k:=f(a+S(k),b+k)$.
The third expectation vanishes since $V_{p(c)}(a+S(T_{a,b}),b+T_{a,b})=0$.
Rearranging this, we obtain
\begin{align}
\label{eq:main_decomp}
\nonumber
&\E\left[{\overline{V}_{p(c)}(a+S(n),b+n){;}\,{T_{a,b}> n}}\right]\\
\nonumber
&\hspace{1cm}=\overline{V}_{p(c)}(a,b)-\EW{\sum\limits_{k=0}^{n-1}f_k{;}\,{T_{a,b}> n}}-\EW{\sum\limits_{k=0}^{T_{a,b}-1}f_k{;}\,{T_{a,b}\leq n}}\\
	&\hspace{1cm}=\overline{V}_{p(c)}(a,b)-\EW{\sum\limits_{k=0}^{T_{a,b}-1}f_k}+\EW{\sum\limits_{k=n}^{T_{a,b}-1}f_k{;}\,{T_{a,b}> n}}.
\end{align}
A priori, existence of the limit of the expectation on the left hand side is not clear:

\begin{proposition}\label{lem:existence}
	The limit  $\lim_{n\to\infty}\E\left[{V_{p(c)}(a+S(n),b+n){;}\,{T_{a,b}> n}}\right]$ is finite.
\end{proposition}
\begin{proof}We first notice that 
$$
\E\left[{V_{p(c)}(a+S(n),b+n){;}\,{T_{a,b}> n}}\right]
=\E\left[{\overline{V}_{p(c)}(a+S(n),b+n){;}\,{T_{a,b}> n}}\right].
$$
It is then immediate from the representation \eqref{eq:main_decomp} that the claim will be proven if we show that the random variable 
$\sum\limits_{k=0}^{T_{a,b}-1}|f_k|$ has finite expectation. Indeed, this will imply that the second term in \eqref{eq:main_decomp} is well-defined and that the third term will vanish as $n\to\infty$.

	For $p(c)\geq 3$, we notice that by Lemma \ref{lem:remainder} 
	$$
	|f(x,y)|\leq C(1+x)^{p(c)-3}+C\frac{(1+x)^{p(c)-1}}{(1+x-c\sqrt{y})^{p(c)-1}}
	$$
	for some suitable positive constant $C$.
Therefore,
\begin{align}
\label{eq:prop.1}
\nonumber
	&\sum\limits_{k=0}^{T_{a,b}-1}|f_k|
	\leq C\sum\limits_{k=0}^{T_{a,b}-1}(a+1+S(k))^{p(c)-3}\\
	&\hspace{3cm}+C\sum\limits_{k=0}^{T_{a,b}-1}
	\frac{(1+a+S(k))^{p(c)-1}}{(1+a+S(k)-c\sqrt{k+b})^{p(c)-1}}.
\end{align}
For the first sum we have
\begin{align*}
\sum\limits_{k=0}^{T_{a,b}-1}(a+1+S(k))^{p(c)-3}
	&\leq C_1\left[(a+1)^{p(c)-3}T_{a,b}+\sum\limits_{k=0}^{T_{a,b}-1}(S(k))^{p(c)-3}\right]\\
	&\leq C_1\left[(a+1)^{p(c)-3}+(\max\limits_{k\leq T_{a,b}}|S(k)|)^{p(c)-3}\right]T_{a,b}
\end{align*}
for some positive constant $C_1$.
We define $\bar S_T:= \max\limits_{k\leq T_{a,b}}|S(k)|$.
Taking the expectation and applying the Cauchy-Schwarz inequality, we get  
\begin{align*}
	 &\EW{\sum\limits_{k=0}^{T_{a,b}-1}(a+1+S(k))^{p(c)-3}}\\
	 &\leq C_1\left[(a+1)^{p(c)-3}\E T_{a,b}+\EW{(\bar S_{T})^{p(c)-3}T_{a,b}}\right]\\
 	&\leq C_1\left((a+1)^{p(c)-3}\E T_{a,b}+\EW{(\bar S_{T})^{(p(c)-3)r}}^{1/r}\EW{T_{a,b}^s}^{1/s}\right),
\end{align*}
where we choose $s=p_1/2<p(c)/2$ and $r=p_1/(p_1-2)$ which satisfies $(p(c)-3)r<p_1$ provided that $p(c)-1<p_1<p(c)$. Applying now Lemma \ref{lem:BDG},
we conclude that
\begin{align}
\label{eq:prop.2}
	 &\EW{\sum\limits_{k=0}^{T_{a,b}-1}(a+1+S(k))^{p(c)-3}}\le
	 C(1+a)^{p(c)-1}.
\end{align}

We now turn to expectation of the second sum in \eqref{eq:prop.1}. 
We start by the following splitting: 
\begin{align*}
&\EW{\sum\limits_{k=0}^{T_{a,b}-1}\frac{(1+a+S(k))^{p(c)-1}}{(1+a+S(k)-c\sqrt{k+b})^{p-1}}}\\
&=\sum\limits_{k=0}^\infty\EW{\frac{(1+a+S(k))^{p(c)-1}}{(1+a+S(k)-c\sqrt{k+b})^{p(c)-1}}{;}\,T_{a,b}>k}\\
&=\sum\limits_{k=0}^\infty\EW{\frac{(1+a+S(k))^{p(c)-1}}{(1+a+S(k)-c\sqrt{k+b})^{p(c)-1}}{;}\,a+S(k)>\left(c+\frac{1}{2}\right)\sqrt{k+b},T_{a,b}>k}\\
	&\hspace{1mm}+ \sum\limits_{k=0}^\infty\EW{\frac{(1+a+S(k))^{p(c)-1}}{(1+a+S(k)-\sqrt{k+b})^{p(c)-1}}{;}\,a+S(k)\leq\left(c+\frac{1}{2}\right)\sqrt{k+b},T_{a,b}>k}.
\end{align*}
We notice that if $a+S(k)>(c+1/2)\sqrt{k+b}$, then 
\[\frac{(1+a+S(k))^{p(c)-1}}{(1+a+S(k)-c\sqrt{k+b})^{p(c)-1}}\leq C.\]
Therefore, applying Lemma~\ref{lem:BDG}, we obtain
\begin{align}
\label{eq:prop.3}
\nonumber
&\sum\limits_{k=0}^\infty\EW{\frac{(1+a+S(k))^{p(c)-1}}{(1+a+S(k)-c\sqrt{k+b})^{p(c)-1}}{;}\,a+S(k)>\left(c+\frac{1}{2}\right)\sqrt{k+b},T_{a,b}>k}\\
&\le C\sum_{k=0}^\infty \P(T_{a,b}>k)=C\E T_{a,b}\le C(1+a)^2.
\end{align}

We now split 
\begin{align*}
\E&\left[{\frac{(1+a+S(k))^{p(c)-1}}{(1+a+S(k)-c\sqrt{k+b})^{p(c)-1}}{;}\,a+S(k)\leq(c+1)\sqrt{k+b},T_{a,b}>k}\right]\notag\\
&\leq C\sqrt{k+b}^{p(c)-1}\sum\limits_{j=1}^{\frac{1}{2}\sqrt{k+b}}
j^{-p(c)+1}\P\left(a+S(k)\in\Delta(j),T_{a,b}>k\right),
\end{align*}
where
$$
\Delta(j)=\left(c\sqrt{k+b}+j-1,c\sqrt{k+b}+j\right].
$$
Using now Lemma~\ref{lem:loc.prob2}, we get 
\begin{align}
\label{eq:prop.4}
\nonumber
&\sum\limits_{j=1}^{\frac{1}{2}\sqrt{k+b}}
j^{-p(c)+1}\P\left(a+S(k)\in\Delta(j),T_{a,b}>k\right)\\
&\hspace{2cm}\le \frac{C(1+a)^{p_1}}{k^{p_1/2+3/4}} \sum\limits_{j=1}^{\frac{1}{2}\sqrt{k+b}}
j^{-p(c)+1}(1+j)^{1/2}.
\end{align}
Since $p(c)\ge3$ the sum on the right hand side of \eqref{eq:prop.4} is bounded
and, consequently,
\begin{align}
\label{eq:prop.5a}
\nonumber
&\sum\limits_{k=0}^\infty\EW{\frac{(1+a+S(k))^{p(c)-1}}{(1+a+S(k)-c\sqrt{k+b})^{p(c)-1}}{;}\,a+S(k)\leq(c+\frac{1}{2})\sqrt{k+b},T_{a,b}>k}\\
&\le C(1+a)^{p(c)-1}+ C(1+a)^{p_1}\sum_{k=1}^\infty k^{(p(c)-1)/2-p_1/2-3/4}.
\end{align}
Therefore, for every $p_1\in(p(c)-1/2,p(c))$ we have  
\begin{align}
\label{eq:prop.5}
\nonumber
&\sum\limits_{k=0}^\infty\EW{\frac{(1+a+S(k))^{p(c)-1}}{(1+a+S(k)-c\sqrt{k+b})^{p(c)-1}}{;}\,a+S(k)\leq(c+\frac{1}{2})\sqrt{k+b},T_{a,b}>k}\\
&\le C(1+a)^{p_1}.
\end{align}
Combining \eqref{eq:prop.4} and \eqref{eq:prop.5}, we finally get 
\begin{equation}
\label{eq:prop.6}
\EW{\sum\limits_{k=0}^{T_{a,b}-1}\frac{(1+a+S(k))^{p(c)-1}}{(1+a+S(k)-c\sqrt{k+b})^{p(c)-1}}}
\le C(1+a)^{p(c)_1}.
\end{equation}
This finishes the proof in the case $p(c)\ge 3$.

In the case $p(c)\in(2,3)$ we have from Lemma~\ref{lem:remainder}
$$
|f_k|\le C\frac{(1+a+S(k))^{p(c)-1}}{(1+a+S(k)-c\sqrt{k+b})^{p(c)-1}}.
$$
Thus, we need to prove an analogue of \eqref{eq:prop.6} for $p(c)\in(2,3)$.
Clearly, \eqref{eq:prop.3} holds for all $p(c)>2$. Furthermore,
it is easy to see that if $p(c)>2.5$ then the sum in \eqref{eq:prop.4} is still bounded and, consequently, \eqref{eq:prop.5} remains valid for $p(c)>2.5$. Assume now that $p(c)\le 2.5$. In this case 
$$
\sum\limits_{j=1}^{\frac{1}{2}\sqrt{k+b}}
j^{-p(c)+1}(1+j)^{1/2}
\le Ck^{5/4-p(c)/2}
$$
and, consequently, instead of \eqref{eq:prop.5a} we have 
\begin{align*}
&\sum\limits_{k=0}^\infty\EW{\frac{(1+a+S(k))^{p(c)-1}}{(1+a+S(k)-c\sqrt{k+b})^{p-1}}{;}\,a+S(k)\leq(c+\frac{1}{2})\sqrt{k+b},T_{a,b}>k}\\
&\le C(1+a)^{p(c)-1}+ C(1+a)^{p_1}\sum_{k=1}^\infty k^{(p(c)-1)/2-p_1/2-3/4+5/4-p(c)/2}\\
&\le C(1+a)^{p(c)-1}+ C(1+a)^{p_1}\sum_{k=1}^\infty k^{-p_1/2}.
\end{align*}
Now we can infer that \eqref{eq:prop.6} remains valid in the case $p(c)\in(2,3)$ with $p_1\in(2,p(c))$.

The case $p(c)\in(1,2]$ is rather similar to $p(c)\in(2,3)$ and we omit its proof.

We now consider the remaining case $p(c)<1$. According to Lemma~\ref{lem:remainder01},
\begin{align*}
|f_k|\le C(1+a+S(k)-c\sqrt{b+k})^{p(c)-2-\delta}. 
\end{align*}
Therefore,
\begin{align*}
\EW{\sum\limits_{k=0}^{T_{a,b}-1}|f_k|}
&\le\EW{C\sum\limits_{k=0}^{T_{a,b}-1}
	(1+a+S(k)-c\sqrt{b+k})^{p(c)-2-\delta}}\\
&=C\sum_{k=0}^\infty\EW{(1+a+S(k)-c\sqrt{b+k})^{p(c)-2-\delta};T_{a,b}>k}.	
\end{align*}
Here we again split every expectation into two parts:
\begin{align*}
&\EW{(1+a+S(k)-c\sqrt{b+k})^{p(c)-2-\delta};T_{a,b}>k}\\
&=\EW{(1+a+S(k)-c\sqrt{b+k})^{p(c)-2-\delta};a+S(k)\le (c+1/2)\sqrt{b+k}, T_{a,b}>k}\\
&+\EW{(1+a+S(k)-c\sqrt{b+k})^{p(c)-2-\delta};a+S(k)> (c+1/2)\sqrt{b+k},T_{a,b}>k}.
\end{align*}
The second expectation is bounded by $k^{p(c)/2-1-\delta/2}\P(T_{a,b}>k)$.
This implies that 
\begin{align}
\label{eq:prop.7}
\nonumber
&\sum_{k=0}^\infty \EW{(1+a+S(k)-c\sqrt{b+k})^{p(c)-2-\delta};a+S(k)> (c+1/2)\sqrt{b+k},T_{a,b}>k}\\
\nonumber
&\le C \sum_{k=0}^\infty k^{p(c)/2-1-\delta/2}\P(T_{a,b}>k)\\
&\le C_1\E T_{a,b}^{p(c)/2-\delta/2}
\le C_2(1+|a|^{p(c)-\delta}+b^{p(c)/2-\delta/2}).
\end{align}
In the last step we have used Lemma~\ref{lem:BDG}.
For every $k$ we also have 
\begin{align*}
&\EW{(1+a+S(k)-c\sqrt{b+k})^{p(c)-2-\delta};a+S(k)\le (c+1/2)\sqrt{b+k}, T_{a,b}>k}\\
&\le \sum_{j=1}^{\sqrt{b+k}/2}j^{p(c)-2-\delta}
\P(a+S(k)\in\Delta(j),T_{a,b}>k).
\end{align*}
Applying now \eqref{eq:lp2}, we have 
\begin{align*}
&\EW{(1+a+S(k)-c\sqrt{b+k})^{p(c)-2-\delta};a+S(k)\le (c+1/2)\sqrt{b+k}, T_{a,b}>k}\\
&\le C\frac{(1+|a|^{p_1}+b^{p_1/2})}{k^{p_1/2+1}}
\sum_{j=1}^{\sqrt{b+k}/2}j^{p(c)-1-\delta}
\le C\frac{(1+|a|^{p_1}+b^{p_1/2})}{k^{p_1/2+1}}k^{p(c)/2-\delta/2}.
\end{align*}
The sequence on the right hand side is summable for any $p_1\in(p(c)-\delta,p(c))$. Thus,
\begin{align*}
&\sum_{k=0}^\infty \EW{(1+a+S(k)-c\sqrt{b+k})^{p-2-\delta};a+S(k)\le (c+1/2)\sqrt{b+k},T_{a,b}>k}\\
&\le C (1+|a|^{p_1}+b^{p_1/2}).
\end{align*}
Combining this with \eqref{eq:prop.7} completes the proof of the proposition.
\end{proof}
As the expectation exists, we can use dominated convergence to take the limit $n\to\infty$. Thus, we obtain
\[\lim\limits_{n\to\infty}\E\left[{V_{p(c)}(a+S(n),b+n){;}\,{T_{a,b}> n}}\right]=V_{p(c)}(a,b)-\EW{\sum\limits_{k=0}^{T_{a,b}-1}f_k}=:W(a,b)\]

\begin{lemma}\label{lem:harmonicity}
	Under the assumptions of Theorem \ref{thm:space-time-h},
	\[\EW{W(a+X,b+1){;}\,T_{a,b}>1}=W(a,b).\]
\end{lemma}

\begin{proof}
	Let $S^*(k):=S(k+1)-X_1$. Clearly, $S(k)$ and $X_1$ are independent, and $S(k)$ and $S^*(k)$ are identically distributed. Let $T_{a,b}^*:=\inf\{k\geq 0:S^*(k)+a<c\sqrt{k+b}\}$.
	We notice first that  
\begin{align*}
&\{T_{a,b}>n+1\}\\
        &=\{a+S(k)>c\sqrt{k+b}\text{ for all }k\le n+1\}\\
		&=\{a+S(1)>c\sqrt{1+b}\}\cap\{X_1+a+(S(k)-X_1)>c\sqrt{k+b}\text{ for all }2\le k\le n+1\}\\
		&=\{T_{a,b>1}\}\cap\{X_1+a+S^*(k)>c\sqrt{k+1+b}\text{ for all }k\le n\}\\
		&=\{a+S(1)>c\sqrt{1+b}\}\cap\{T^*_{a+X_1,b+1}>n\}.
\end{align*}
	Therefore, using law of total expectation to condition on the first step of the random walk,
	\begin{align*}
		\E&\left[V_{p(c)}(a+S(n+1),b+n+1){;}\,T_{a,b}>n+1\right]\\
		&=\E\left[\E\left[V_{p(c)}(a+S(n+1),b+n+1){;}\,T_{a,b}>n+1\left|X_1\right.\right]\right]\\
		&=\E\left[\E\left[V_{p(c)}(a+X_1+S^*(n),b+1+n){;}\,T^*_{a+X_1,b+1}>n\left|X_1\right.\right]{;}\,T_{a,b}>1\right]\\
		&=\int\E\left[V_{p(c)}(a+x+S(n),b+1+n){;}\,T_{a+x,b+1}>n\right]\P(X_1\in dx, T_{a,b}>1)\\
		&=\int\E\left[V_{p(c)}(y+S(n),b+1+n){;}\,T_{y,b+1}>n\right]\P(a+X_1\in dy, T_{a,b}>1)
	\end{align*}
	We have seen in the proof of Lemma \ref{lem:existence} that 
    $\E\left[V_{p(c)}(y+S(n),b+1+n){;}\,T_{y,b+1}>n\right]$ is bounded by 
    $C(1+|y|^p)$. Then, due to the finiteness of $\E|X|^p$, we may apply dominated convergence theorem:
	\begin{align*}
		W(a,b)&=\lim\limits_{n\to\infty}\E\left[V_{p(c)}(a+S(n+1),b+n+1){;}\,T_{a,b}>n+1\right]\\
		&=\lim\limits_{n\to\infty}\int\E\left[V_{p(c)}(y+S(n),b+1+n){;}\,T_{y,b+1}>n\right]\P(a+X_1\in dy, T_{a,b}>1)\\
		&=\int W(y,b+1)\P(a+X_1\in dy, T_{a,b}>1)\\
		&=\EW{W(a+X_1,b+1){;}\,T_{a,b}>1}.\qedhere
	\end{align*}
\end{proof}

\begin{lemma}\label{lem:equiv.W.V}
    The function $W(a,b)$ is monotone increasing in $a$. Furthermore,
	if $b\to\infty$ with $a>(c(p)+\gamma)\sqrt{b}$ for some $\gamma>0$, 
    then
\begin{equation}\label{eq:equiv.W.V.}
    W(a,b)\sim V_{p(c)}(a,b).
    \end{equation}
    Moreover, we can pick $\gamma_b\downarrow 0$ sufficiently slow 
    so that~\eqref{eq:equiv.W.V.} still  holds for $a>(c(p)+\gamma_b)\sqrt{b}$ . 
\end{lemma}
\begin{proof}
    Monotonicity of $a\mapsto W(a,b)$ is rather obvious. It is immediate from the observation that $V_{p(c)}(a,b)$ and $T_{a,b}$ are increasing in $a$.

	If $a>(c(p)+\gamma)\sqrt{b}$, then $V_{p(c)}(a,b)\geq C(a-c\sqrt{b})^{p(c)}$ for some suitable constant $C,C'$. It suffices thus that in every case, $\EW{\sum\limits_{k=0}^{T_{a,b}-1}|f_k|}$ is of lower order.
	
	This statement follows from the proof of Proposition \ref{lem:existence}: there we have shown that if $p(c)>1$ then $\EW{\sum\limits_{k=0}^{T_{a,b}-1}|f_k|}=O((1+|a|)^{p_1})$ for an appropriately chosen $p_1<p(c)$. And for $p(c)<1$ we have proven that $\EW{\sum\limits_{k=0}^{T_{a,b}-1}|f_k|}=O((1+|a|^{p_1}+b^{p_1/2}))$ for an appropriate $p_1<p(c)$. In both cases this term is of smaller order than $(a-c\sqrt{b})^{p(c)}$. Thus, the proof is complete.
\end{proof}
\begin{proof}[Proof of Corollary~\ref{cor:UBound}]
Recalling that $a\mapsto W(a,b)$ is increasing and using Lemma~\ref{lem:probprodbound}, we obtain 
$$
\P(T_{a,b}>n)\le\frac{\E[W(a+S(n),b+n);T_{a,b}>n]}{\E[W(a+S(n),b+n);a+S(n)>c\sqrt{b+n}]}.
$$
Then, by the harmonicity of $W$,
$$
\P(T_{a,b}>n)\le\frac{W(a,b)}{\E[W(a+S(n),b+n);a+S(n)>c\sqrt{b+n}]}.
$$
Thus, it remains to derive a lower bound for the expectation in the denominator.

We know from Theorem~\ref{thm:space-time-h} that $W(a,b)\sim V_{p(c)}(a,b)$
as $a>(c+1)\sqrt{b}$ and $a\to\infty.$ Combining this with Lemma~\ref{lem:V-prop}, we conclude that 
\begin{align*}
&\E[W(a+S(n),b+n);a+S(n)>c\sqrt{b+n}]\\
&\hspace{1cm}
\ge A (b+n)^{p(c)/2}\P(a+S(n)>(c+1)\sqrt{b+n})
\end{align*}
for some constant $A$, which does not depend on $a$ and $b$.
Since $a>c\sqrt{b}$ and $n\ge b$,
\begin{align*}
\P(a+S(n)>(c+1)\sqrt{b+n})&\ge \P(S(n)>(c+1)\sqrt{b+n}-c\sqrt{b})\\
&\ge \P(S(n)>(|c|+1)\sqrt{2n}).
\end{align*}
The desired bound follows now from the central limit theorem.
\end{proof}

\section{Asymptotics for \texorpdfstring{$\P(T_{a,b}>n)$}{thm.tail}: proof of Theorem~\ref{thm:tail}.}
For all $k<n$ we define 
\begin{align*}
Q_{k,n}(y)
&:=\P(y+S(j)>g(k+j)\text{ for all }0\le j\le n-k)\\
&=\P(y+S(j)>c\sqrt{b+k+j}\text{ for all }0\le j\le n-k),\\
&=\P(T_{a+y, b+k}>n-k).
\end{align*}

Applying Corollary~\ref{cor:UBound} 
we conclude that
$$
Q_{k,n}(y)\le M \frac{W(a+y,b+k)}{(n-k)^{p(c)/2}}, \quad k\le \frac{n-b}{2}.
$$
Therefore, for every $k\le \frac{n-b}{2}$,
\begin{align}
\label{eq:bp3}
\left|(n-k)^{p(c)/2}Q_{k,n}(y)-\varkappa W(a+y,b+k)\right|
\le(M+\varkappa)W(a+y,b+k).
\end{align}

Let $\gamma_n$ be a sequence, which converges to zero sufficiently slow. 
Choose now a further sequence $\mu_n$ such that $\mu_n\to0$ and $\frac{\mu_n}{\gamma_n}\to\infty$. If $\gamma_n\to0$ sufficiently slow then, using the functional central limit theorem, we get, uniformly 
in $(y,k)\in I_n$, where 
$ I_n:=\{(y,k)\,\colon\, k\le (n-b)/2,\,a+y-c\sqrt{b+k}\in(\gamma_n\sqrt{n},\mu_n\sqrt{n})\}$, 
$$
\P(T_{a+y,b+k}>n-k) \sim \P(T^{(bm)}_{a+y,b+k}>n-k). 
$$
Then, using Uchiyama~\cite{U80},  
\[
\left|\P(T_{a+y,b+k}>n-k)-\varkappa\frac{V_{p(c)(a+y,b+k)}}{(n-k)^{p(c)/2}}\right|
\le \varepsilon_n \frac{V_{p}(a+y,b+k)}{(n-k)^{p(c)/2}},
\]
uniformly in  $(y,k)\in I_n$ 
for some $\varepsilon_n\downarrow 0$. 
We know from Lemma~\ref{lem:equiv.W.V}
that 
$$
V_{p(c)}(a+y,b+k)\sim W(a+y,b+k)
$$
uniformly in $(y,k)\in I_n$.
Consequently,
\begin{align}
\label{eq:bp4}
\left|(n-k)^{p(c)/2}Q_{k,n}(y)-\varkappa W(a+y,b+k)\right|
\le 2\widetilde{\varepsilon}_nW(a+y,b+k)
\end{align}
uniformly $(y,k)$ such that $k\le n/2$ and 
$y\in I_n$.

We conclude then
\begin{align}
 \label{eq:bp6}
 \nonumber 
 &\left|(n-k)^{p(c)/2}Q_{k,n}(y)-\varkappa W(a+y,b+k)\right|\\
 &\hspace{0.5cm}\le 
 2\widetilde{\varepsilon}_nW(a+y,b+k)
 +(M+\varkappa)W(a+y,b+k){\rm 1}\{(y,k)\notin I_n\}.
 \end{align}

For every $m<n$ we define 
$$
\nu_m:=\inf\{k\ge1: S(k)-g_{a,b}(k)>\sqrt{m}\}\wedge m.
$$

By the strong Markov property at time $\nu_m$,
$$
\P(T_{a,b}>n)
=\sum_{k=1}^m\int_{\R}\P(S_k\in dy,\nu_m=k<T_{a,b})Q_{k,n}(y).
$$
Applying \eqref{eq:bp6}, we have 
\begin{align*}
&\left|n^{p(c)/2}\P(T_{a,b}>n)
-\varkappa\E[W(a+S(\nu_m),b+\nu_m);T_{a,b}>\nu_m]\right|\\
&\hspace{1.5cm}
+2\widetilde{\varepsilon}_n\E[W(a+S(\nu_m),b+\nu_m);T_{a,b}>\nu_m]\\
&\hspace{2cm}
+C\E[W(a+S(\nu_m),b+\nu_m); T_{a,b}>\nu_m, (S(\nu_m),\nu_m)\notin I_n].
\end{align*}
Since $\nu_m$ is a bounded stopping time,
\[
\E[W(a+S(\nu_m),b+\nu_m);T_{a,b}>\nu_m] = W(a,b)  
\]
and, consequently, 
\begin{multline}
\label{eq:tail1}
\left|n^{p(c)/2}\P(T_{a,b}>n)
-(\varkappa+o(1))W(a,b)\right|\\ 
\le 
C\E[W(a+S(\nu_m),b+\nu_m);T_{a,b}>\nu_m, (S(\nu_m),\nu_m)\notin I_n]. 
\end{multline}
We are left to show that the right-hand side converges to $0$ for an appropriately chosen sequence $m=m(n)$.

\begin{lemma}\label{lem26}
There exists a constant $C$ such that 
\[
\P (T_{a,b}>\nu_m) \le C\frac{W(a,b)}{m^\frac{p(c)}{2}}, \quad \text{for all } m\ge b.
\]
\end{lemma}
\begin{proof}
It is immediate from the definition of $\nu_m$ that 
\begin{align}
\label{eq:Wprop6}
\nonumber 
\P(T_{a,b}>\nu_m)
&=\P(T_{a,b}>\nu_m=m)+\P(T_{a,b}>\nu_m,\nu_m<m)\\
&\le\P(T_{a,b}>m)+\P(T_{a,b}>\nu_m,S(\nu_m)>g_{a,b}(\nu_m)+\sqrt{m}). 
\end{align}
According to Corollary~\ref{cor:UBound}, 
\[
\P(T_{a,b}>m)\le M\frac{W(a,b)}{m^{p(c)/2}},
\]
for some constant $M.$ 

If $y>g_{a,b}(k)+\sqrt{m}$ and $k\le m$ then 
\begin{align*}
W(a+y,b+k)&\ge W(a+g_{a,b}(k),b+k)
=W(c\sqrt{b+k}+\sqrt m, b+k)\\
&\ge C 
V_{p(c)}(c\sqrt{b+k}+
\sqrt{m})\ge C
m^{p(c)/2}
\end{align*}
for some positive constant $C$. Using this fact and applying the Markov inequality, we have 
\begin{align*}
 \P(T_{a,b}>\nu_m,S(\nu_m)>g_{a,b}(\nu_m)+\sqrt{m})
 \le C\frac{\E[W(a+S(\nu_m),b+\nu_m);T_{a,b}>\nu_m]}{m^{p(c)/2}}.
\end{align*}
Since $W$ is harmonic we further obtain that 
\[
 \P(T_{a,b}>\nu_m,S(\nu_m)>g_{a,b}(\nu_m)+\sqrt{m})
 \le C\frac{\E[W(a,b)}{m^{p(c)/2}}.
\]
\end{proof}
\begin{lemma}
\label{lem:W-tail}
If $A_m\to\infty$ as $m\to\infty$ then 
\begin{equation}
\label{eq:Wprop9}
\E[W(a+S(\nu_m),b+\nu_m);T_{a,b}>\nu_m, S(\nu_m)>A_m\sqrt{m}]
=o(1).
\end{equation}
\end{lemma}
\begin{proof}
Combining Lemma~\ref{lem:V-prop} and Lemma~\ref{lem:equiv.W.V} we have  
$$
W(x,t)\le C|x|^{p(c)}\quad\text{for all }x>(c+1)\sqrt{t}.
$$
Moreover, by the definition of $\nu_m$,
$$
|S(\nu_m)|\le 
c\sqrt{\nu_{m-1}+b}-a+\sqrt{m} +X_{\nu_m}
\le 
(|c|+2)\sqrt{m}+X_{\nu_m}
$$
for all sufficiently large values of $m$. 
This implies that there exists $m_0$ such that 
$\{S(\nu_m)>A_m\sqrt{m}\}\subset\left\{X_{\nu_m}>\frac{A_m}{2}\sqrt{m}\right\}$
and $W(a+S(\nu_m),b+\nu_m)\le C(X_{\nu_m})^{p(c)}$ for all $n\ge n_0$. Combining these estimates with the strong Markov property, we obtain 
\begin{align*}
&\E[W(a+S(\nu_m),b+\nu_m);T_{a,b}>\nu_m, S(\nu_m)>A_m\sqrt{m}]\\
&\hspace{1cm}
\le \sum_{k=1}^n\P(T_{a,b}>k-1)\E[X^{p(c)};X>A_m\sqrt{m}/2]. 
\end{align*}
When $p(c)>2$ then by Corollary~\ref{cor:UBound} expectation   $\E[T_{a,b}]$  is finite and~\eqref{eq:Wprop9} follows then from the Markov inequality.

Assume now that $p(c)\le2$. Since in this case we additionally assume that
$\E|X|^{2+\delta}$ is finite it follows that  
$\E[X^{p(c)};X>A_m\sqrt{m}/2]=o(m^{p(c)/2-1-\delta/2})$. 
Using once again Corollary~\ref{cor:UBound} we see that 
\[
\sum_{k=1}^m \P(T_{a,b}>k-1) \le C m^{1-\frac{p(c)}{2}}. 
\]
Combining these two estimates we see that~\eqref{eq:Wprop9} holds for $p(c)\le 2$ as well.  
\end{proof}

We can now complete the proof of Theorem~\ref{thm:tail}
by showing that
\[
E_n:=\E[W(a+S(\nu_m), b+\nu_m);T_{a,b}>\nu_m, (S(\nu_m),\nu_m)\notin I_n]
=o(1). 
\]
We have, $E_n\le E_{n,1}+E_{n,2}$, where 
\begin{align*}
E_{n,1}& =  
\E[W(a+S(\nu_m), b+\nu_m);T_{a,b}>\nu_m, a+S(\nu_m)\ge \nu_n\sqrt n -|c|\sqrt{b+m}],\\
E_{n,2}&=
\E[W(a+S(\nu_m), b+\nu_m);T_{a,b}>\nu_m, a+S(\nu_m)-c\sqrt{b+\nu_m}\le \gamma_n\sqrt n].
\end{align*}
When $\frac{\mu_n \sqrt n}{\sqrt m}\to \infty$ the first expectation $E_{n,1}=o(1)$ by Lemma~\ref{lem:W-tail}.  

Since $W(y,t)$ is increasing in $y$, 
\[
E_{n,2} \le 
\E[W(\gamma_n \sqrt n +c\sqrt{b+\nu_m}, b+\nu_m);T_{a,b}>\nu_m]. 
\]
Using Lemma~\ref{lem:equiv.W.V} and $V_{p(c)}(x,t)\le C(x-c\sqrt{t})^{p(c)}$ we obtain 
\[
E_{n,2} \le C(\gamma_n \sqrt{n})^{p(c)} \P(T_{a,b}>\nu_m).
\]
Applying now Lemma~\ref {lem26}
\[
E_{n,2}\le C\left(\frac{\gamma_n \sqrt n}{\sqrt m}\right)^{p(c)} =o(1) 
\]
for $m=m(n)$ such that $\frac{m(n)}{\gamma_n^2 n}\to \infty$.
Thus, $E_n\to0$ if $m(n)$ is such that $\frac{m(n)}{\gamma_n^2 n}\to \infty$ and $\frac{m(n)}{\mu_n^2 n}\to 0$. This completes the proof of Theorem~\ref{thm:space-time-h}.


    \end{document}